\documentclass[12pt]{amsart}
\usepackage{amsthm}
\usepackage{hyperref}
\usepackage{a4wide}
\usepackage{array, xcolor}
\usepackage{amsfonts}
\usepackage{amsmath}
\usepackage{mathtools}
\usepackage{amssymb}
\usepackage{graphicx}
\usepackage{float}
\usepackage{accents}
\usepackage{verbatim}
\newtheorem{proposition}{Proposition}
\newtheorem{lemma}[proposition]{Lemma}

\newtheorem{theorem}[proposition]{Theorem}
\usepackage{tikz}
\usetikzlibrary{matrix}

\theoremstyle{definition}
\newtheorem{definition}[proposition]{Definition}

\newtheorem*{remark*}{Remark}

\def\RR{{\mathbb R}}
\def\Rn{{\mathbb R^n}}

\def\ZZ{{\mathbb Z}}
\def\NN{{\mathbb N}}
\def\CC{{\mathbb C}}
\def\HH{{\mathbb{H}}}
\def\DD{{\mathbb{D}}}
\def\g{{\gamma}}
\def\G{{\Gamma}}

\DeclareMathOperator{\tr}{Tr}

\DeclareMathOperator{\D}{D}
\DeclareMathOperator{\PP}{P}

\DeclareMathOperator{\psl}{PSL_2(\mathbb{R})}

\DeclareMathOperator{\F}{\Phi}
\DeclareMathOperator{\area}{Area}

\DeclareMathOperator{\spec}{Spec}
\DeclareMathOperator{\End}{End}
\DeclareMathOperator{\Res}{Res}

\newcommand{\cl}{\operatorname{Cl}}
\newcommand{\li}{\operatorname{Li}}
\newcommand{\gl}{\operatorname{GL}}
\newcommand{\slinear}{\operatorname{SL}}
\newcommand{\diam}{\operatorname{Diam}}
\newcommand{\so}{\operatorname{SO}}
\newcommand{\spin}{\operatorname{Spin}}

\newcommand{\Id}{\operatorname{Id}}

\begin{document}
\title[Trace formula for Dirac operators on degenerating hyperbolic surfaces]{The Selberg trace formula for spin Dirac operators on degenerating hyperbolic surfaces}
\author{Rare\c s Stan}
\address{Institute of Mathematics of the Romanian Academy\\ 
Bucharest\\ 
Romania}
\email{rares.stan@imar.ro}

\begin{abstract}
We investigate the spectrum of the spin Dirac operator on families of hyperbolic surfaces where a set of disjoint simple geodesics shrink to $0$, under the hypothesis that the spin structure is non-trivial along each pinched geodesic. The main tool is a trace formula for the Dirac operator on finite area hyperbolic surfaces. We derive a version of Huber's theorem and a non-standard small-time heat trace asymptotic expansion for hyperbolic surfaces with cusps. As a corollary we find a simultaneous Weyl law for the eigenvalues of the Dirac operator which is \emph{uniform} in the degenerating parameter. The main result is the convergence of the Selberg zeta function associated to the Dirac operator on such families of hyperbolic surfaces. A central role is played by a $\{ \pm 1 \}$-valued class function $\varepsilon$ determined by the spin structure.
\end{abstract}

\date{\today}
\maketitle

\section{Introduction}
This paper is concerned with the behaviour of the spectrum of the spin Dirac operator on compact hyperbolic surfaces during a deformation process towards a hyperbolic surface with cusps.
B\" ar \cite{Bar} discovered that the Dirac operator on a finite volume surface has discrete spectrum if and only if the spin structure is non-trivial along all cusps. Under this assumption, he showed uniform estimates for the number of eigenvalues in a fixed interval $(-\xi,\xi)$ for a sequence of compact surfaces degenerating to a surface with cusps. 
Moroianu \cite{Moroianu} extended B\" ar's result to a large class of open manifolds equipped with so-called cusp metrics, a general class of metrics which includes complete hyperbolic manifolds of finite volume. Under the hypothesis that the Dirac operator on the section of a cusp is invertible, he showed that the spectrum is discrete and, notably, obeys the Weyl law.

Here we first prove a Selberg trace formula for finite volume hyperbolic surfaces endowed with a spin structure obeying B\" ar's non-triviality condition, extending the formula obtained by Bolte and Stiepan for the compact case \cite{BolteStiepanSelbergForDirac}. This formula extends Huber's isospectrality result to finite volume hyperbolic surfaces. It is natural to ask what happens to the Selberg trace formula for such a family of degenerating surfaces. We found that the geometric side converges in the following sense: the primitive conjugacy classes corresponding to the pinched geodesics contribute a $\log(2)$, and the rest converges to the corresponding geometric side in the trace formula on the finite volume limit surface. This has several consequences. First, we find that the small-time heat trace has a non-standard asymptotic term $k\log(2)(4\pi t)^{-1/2}$, where $k$ is the number of cusps. Based on this asymptotics, we prove that the Weyl law is \emph{uniform} for families of compact hyperbolic surfaces degenerating towards a surface with cusps, provided the spin structure is non-trivial along the pinched geodesics, thereby refining \cite[Theorem 2]{Bar}. We then prove the holomorphic extension to $\CC$ and a symmetry with respect to the point $1/2$ for the Selberg zeta function $Z_{\varepsilon}$ of a finite volume surface, defined in $(\ref{formulaZeta})$, with respect to a certain class function $\varepsilon: \G \longrightarrow \{ \pm 1 \}$ induced by the spin structure.

The behaviour under degeneracy of the Selberg zeta function corresponding to the scalar Laplacian(i.e., the hypothetical case $\varepsilon \equiv 1$) was investigated by Schulze \cite{SchulzeZetaLaplacian}, who proved the so-called Wolpert's conjecture in the half-plane $\Re(s)>1/2$. That convergence cannot have any counterpart in the rest of the complex plane because of the appearance of continuous spectrum in the limit. 

The main result in this paper is that, after rescaling by a constant factor depending in a simple way on the lengths of the pinched geodesics, our Selberg zeta function converges on the whole complex plane. The limit of this convergence is an exponential multiplied by the Selberg zeta function associated to the non-compact limit surface with cusps.

In the rest of this introduction we explain in more detail the notions and results above.
\subsection{Degenerating surfaces}
Throughout the paper, an oriented complete hyperbolic surface is a quotient of the Poincar\' e half plane, $\HH:=\left( \{ (x,y)\in \RR^2: y>0 \}, \frac{dx^2+dy^2}{y^2}\right)$, through a discrete subgroup of isometries $\G \subset \psl$. This quotient inherits a complete Riemmanian metric of sectional curvature $-1$, which in turn determines the group $\G$ up to conjugation in $\psl$. To define a \emph{pinching process} consider $\kappa$ disjoint, simple, closed geodesics $\eta_j$, for $1\leq j\leq 3g-3$, on a compact hyperbolic surface $M$ of genus $g$. Complete this to a maximal system of such geodesics, denoted $\eta_j$, for $\kappa+1\leq j\leq3g-3$. Now decompose the surface into pairs of pants, cutting along each $\eta_j$. Take a continuous family of complete hyperbolic metrics $\{g_t\}_{\vert_{t\in (0,1]}}$ on $M$ whose restriction to $M\setminus \cup_{j=1}^{\kappa} \eta_j$ extends continuously to a family $\{g_t\}_{\vert_{t\in [0,1]}}$. We shall denote by $l_t(c)$ the length of the curve $c$ with respect to the metric $g_t$.
\begin{definition}\label{pinchingProcess}
We say that $M$ goes through a \emph{pinching process} along $\{ \eta_j: 1\leq j\leq \kappa \}$ if:
\begin{itemize}
\item $g_1$ equals the initial metric on $M$ and $g_0$ is complete hyperbolic on $M\setminus \cup_{j=1}^{\kappa} \eta_j$;
\item $l_t(\eta_j)\rightarrow 0$ as $t\rightarrow 0$ for $1\leq j\leq \kappa$;
\item $l_t(\eta_j)\rightarrow \alpha_j>0$ as $t\rightarrow 0$ for $\kappa+1\leq j\leq 3g-3$;
\item $\theta_j(t)$ have a limit when $t\rightarrow 0$, for $1\leq j\leq 3g-3$, where $\theta_j(t)$ is the twisting parameter in the Fenchel-Nielsen coordinates on the Teichm\" uller space.
\end{itemize}
\end{definition}
At the limit $t=0$, we obtain a non-compact surface $M\setminus \cup_{j=1}^{\kappa} \eta_j$ with $2\kappa$ cusps, endowed with a complete metric $g_0$. Clearly, the area remains constant throughout this process. 
Near a geodesic $\eta_j$, the metric $g_t$ can be written as: 
\begin{align*}
g_t = \frac{dx^2}{x^2+l_t^2(\eta_j)} + \left( x^2+ l_t^2(\eta_j) \right)d\theta^2;
&&
(x,\theta)\in \left(-\frac{l_t(\eta_j)}{\sinh \frac{l_t(\eta_j)}{2} },\frac{l_t(\eta_j)}{\sinh \frac{l_t(\eta_j)}{2} }\right) \times [0,2\pi].
\end{align*}
Such degeneration processes were described by Ji \cite{lizhen}, B\" ar \cite{Bar}, and Schulze \cite{SchulzeZetaLaplacian}.

\subsection{Trace formula for the Dirac operator}
The main tool we develop is a trace formula for the spin Dirac operator on finite area hyperbolic surfaces (Theorem \ref{selbergTraceFormula}). To obtain it, we use the kernel of a convolution operator which has the same eigenspinors as the squared Dirac, similar with Selberg's original idea. The main difference between this formula and the one for the Laplace operator is the appearance of a \emph{class function} $\varepsilon$. By this we mean a function defined on the fundamental group $\G$ of our surface, constant along conjugacy classes. This function also appeared in the work of D'Hocker and Phong \cite{dhockerPhong}, and Sarnak \cite{SarnkDeterminantsOfLaplacians}. It was used when computing determinants of Laplacians in relation to the Selberg zeta function. The definition of $\varepsilon$ in \cite{dhockerPhong} is rather algebraic. In the second section we prove that this definition can be reinterpreted in the usual, geometric fashion, as arising from the action of elements in $\G$ on spinors. The information about the spin structure will be encapsulated in this class function. We will show that B\" ar's non-triviality condition is equivalent to the fact that $\varepsilon$ takes value $-1$ on each primitive parabolic element. The negative sign produces the alternating harmonic series, hence the $\log (2)$ term in Theorem \ref{selbergTraceFormula}. There is an advantage in regarding the non-triviality condition using $\varepsilon$. If a surface goes through a pinching process along some geodesics, and $\varepsilon$ takes value $-1$ on each one of those geodesics, then the trace formula converges to the trace formula of the limit surface (Theorem \ref{limitSelbergFormula}).
Several applications can be obtained from here.

\subsection{Huber's Theorem for the Dirac operator}
Huber's Theorem \cite{Buser} allows us to tackle isospectrality related problems. It states that two compact Riemann surfaces of genus greater than two have the same Laplace spectrum (i.e.~are \emph{isospectral}) if and only if they have the same \emph{length spectrum} (i.e.~the sequence of lengths of oriented, closed geodesics). If one wants to extend this result to the case of finite area Riemann surfaces one encounters a problem. The spectrum of $\Delta$ becomes continuous and not much is known about the embedded eigenvalues. We address this issue by using the Dirac operator instead, for a non-trivial spin structure. Two surfaces are called \emph{Dirac isospectral} if they have the same Dirac spectrum. In general, two Dirac isospectral manifolds are neither isometric nor homeomorphic. Various examples can be found in the literature, we mention B\" ar \cite{Bar2}, and Ammann and B\" ar \cite{Ammann}. We will not pursue this direction here. Our aim is to prove the following:
\begin{theorem}\label{isospectralityTheorem}
Two finite area hyperbolic surfaces, endowed with non-trivial spin structures are Dirac isospectral if and only if they have the same length spectrum and their respective class functions $\varepsilon$ coincide on the primitive length spectrum.
\end{theorem}

\subsection{Heat trace asymptotics and uniform Weyl law}
The trace formula helps us deduce the small-time heat trace asymptotics. Interestingly, each cusp contributes with a non-standard $T^{-1/2}$ term as $T\rightarrow 0$.
\begin{theorem}\label{heatTraceAsymptotic}
Let $M=\G \setminus \HH$ be a complete hyperbolic surface of finite volume with $k$ cusps, where $\G$ is a discrete subgroup of $\psl$. Suppose that the spin structure is such that $\varepsilon(\g)=-1$ for every primitive parabolic element $\g\in\G$. If $\{\lambda_j\}_{j\in {\mathbb{N}}}$ are the eigenvalues of $\D^-\D^+$ (i.e. half the spectrum of $\D^2$) on $M$, then:
\begin{align*}
\sum_{j=0}^{\infty} e^{-T\lambda_j} \sim \frac{\area (M)}{4 \pi T} -\frac{k \log(2)}{\sqrt{4\pi T}} + \frac{\area (M) }{4\pi} \sum_{m=0}^{\infty}a_m T^m; && \text{as $T\searrow 0$},
\end{align*}
where:
\begin{align*}
a_m:=\int_{\RR}e^{-T\xi^2 }(-1)^m\xi^{2m}\left( \xi\coth(\pi \xi)-|\xi|\right) d\xi.
\end{align*}
\end{theorem}
Note that this asymptotics is purely topological, meaning that it only depends on the number of cusps and the first Betti number but neither on the particular hyperbolic metric nor on the spin structure (provided it is non-trivial). Following the approach of Minakshisundaram and Pleijel \cite{Pleijel}, we will use this asymptotic together with Karamata's theorem to deduce the uniform Weyl law:

\begin{theorem}\label{uniformWeylLaw}
Let $M$ be a compact hyperbolic surface going through a pinching process along $\cup_{j=1}^k \eta_j $ (corresponding to the conjugacy class $[\mu_j]$, for $\mu_j \in \G$). Denote $N_t(\lambda)$ the number of $\D_t^2$-eigenvalues (counted with multiplicity) smaller than $\lambda$. If $\varepsilon(\mu_j)=-1$ for $1\leq j \leq k$ (i.e. the spin structure along $\mu_j$ is non-trivial), then:
\begin{align*}
\lim_{\lambda\rightarrow \infty} \frac{N_t(\lambda)}{\lambda} = \frac{\area (M)}{2\pi};
&&
\text{\emph{uniformly for} } t\in [0,1].
\end{align*}
\end{theorem}
In terms of the eigenvalue counting function for $\D_t$, this result can be rephrased as:
\begin{align*}
N_{D_t}(-\xi,\xi) = \xi^2\frac{\area (M)}{2\pi}+ o(\xi^2),
\end{align*}
uniformly in $t\in [0,1]$.

Note that sharp estimates for the error term in the Weyl law are known (see e.g. Ivrii\cite{ivrii}). We will not address this aspect here.

\subsection{The convergence of the Selberg Zeta function}
For any $s\in \CC$ with $\Re (s)>1$ we define the Selberg Zeta function with respect to the class function $\varepsilon$ as:
\begin{align}\label{formulaZeta}
Z_{\varepsilon}(s, M):=\prod_{[\mu]} \prod_{m=0}^{\infty} \left( 1-\varepsilon(\g) e^{-l(\mu)(s+m)}\right),
\end{align}
where $[\mu]$ runs along all conjugacy classes of hyperbolic, primitive elements $\mu \in \G$. Because the $j$-th longest geodesic is approximately $\log j$ (Lemma $\ref{marginireGeodeziceNonCompact}$), this product is absolutely convergent. It turns out that this function extends holomorphically to the whole complex plane. Moreover, if our surface goes through a pinching process along a geodesic $\eta_t$, we obtain the main result of the paper:
\begin{theorem}\label{convergenceZeta}
Let $M$ be a compact hyperbolic surface going through a pinching  process along $\cup_{j=1}^{\kappa} \eta_j$ (corresponding to the conjugacy class $[\mu_j]$, for $\mu_j \in \G$). If $\varepsilon(\mu_j)=-1$ for $1\leq j \leq \kappa$ then:
\begin{align*}
\lim_{t\rightarrow 0}Z_{\varepsilon}(s,(M,g_t)) \exp\left( - \sum_{j=1}^{\kappa} \frac{ \pi^2}{6l_t(\eta_j)} \right) = Z_{\varepsilon}(s,(M,g_0)) 2^{{\kappa}(1-2s)},
\end{align*}
uniformly on compacts in $\CC$, where $g_t$ is the hyperbolic metric on $M$ at time $t\in [0,1]$. 
\end{theorem}
This result goes further than the analog of Wolpert's conjecture \cite{Wolpert} in our case. Unlike Schulze \cite{SchulzeZetaLaplacian}, who proves the conjecture, we rescale the function by a constant factor which only depends on the length of the pinched geodesic. Thus, up to this rescaling, the family of Selberg zeta functions is continuous along certain paths towards the boundary of the Teichm\" uller space. To prove this result, from the trace formula (Theorem \ref{selbergTraceFormula}), the logarithmic derivative of the Selberg zeta function can be expressed as the trace of a difference of resolvents: $(\D_t^2 - (s-1/2)^2)^{-1} - (\D_t^2 - (s_0-1/2)^2)^{-1}$. Finally, we combine Theorem \ref{uniformWeylLaw} with Theorem 1 of Pf\" affle \cite{Pfaffle}, which in turn is based on the work of Colbois--Curtois \cite{curtois1, curtois2}, to obtain the convergence of the logarithmic derivative of $Z_{\varepsilon}$, from which we derive Theorem \ref{convergenceZeta}.
\subsection*{Acknowledgements} 
I am indebted to Sergiu Moroianu for giving me the idea to study this problem as part of my doctoral research and for introducing me to the world of Selberg trace formulae. I would also like to thank Bernd Ammann, Cipriana Anghel, L\' eo B\' enard, Cezar Joi\c ta, Laura Monk, and Julian Seipel for helpful conversations during the preparation of this paper. Last but not least I would like to thank the anonymous referee for insightful comments and suggestions and for pointing out the references \cite{Hoffmann} and \cite{Warner}.
\subsection*{Funding}
The author was partially supported from the project PN-III-P4-ID-PCE-2020-0794 and from the PNRR-III-C9-2023-I8 grant CF 149/31.07.2023 {\em Conformal Aspects of Geometry and Dynamics}.

\section{The Dirac operator}
\subsection{The spinor bundle}
Given a $n$-dimensional vector space $V$ and a bilinear symmetric form $q$, the Clifford algebra associated to $(V,q)$ is defined as the quotient of the tensor algebra by the Clifford ideal, generated by all the elements of the form $v \otimes w + w\otimes v + 2 q(v,w)$. If $V=\Rn$ and $q$ is the standard scalar product, the associated Clifford algebra is usually denoted $\cl_n$. The spin group $\spin(n)$ is the subgroup of $\cl_n$ consisting of all even products of unit vectors. It follows that $\spin(n)$ is a connected, two-sheeted covering of the special orthogonal group $\so(n)$. Moreover, if $n\geq 3$, $\spin(n)$ becomes the universal cover of $\so(n)$. If $n$ is even, the complexified Clifford algebra $\cl_n \otimes_{\RR} \CC$ is isomorphic to the algebra of complex matrices of size $2^{n/2}$. Let us give a concrete example of this representation. Take $\{ e_1,..., e_{2k} \}$, for $2k=n$, the standard basis in $\RR^{n}$. The vector space $V$ acts on the complex vector space $\Sigma_n:=\wedge^*W$, where $W$ is generated by $\{ 2^{-1/2}(e_1 -ie_2), ..., 2^{-1/2}(e_{2k-1} - ie_{2k}) \}$ by the following action:
\[
cl (v) : = 2^{-1/2}(v-iJ(v))\wedge (\cdot) - 2^{-1/2}(v+iJ(v))\lrcorner (\cdot),
\]
where $J$ is the standard almost complex structure on $\Rn$. Clearly it extends to the tensor algebra $T \CC^n= \oplus_{j=0}^{\infty} (\CC^n)^{\otimes j} $. Since $cl$ vanishes on the Clifford ideal, it descends to $\cl_n\otimes_{\RR}\CC$:
\begin{align*}
cl : \cl_n\otimes_{\RR}\CC \longrightarrow \End_{\CC} (\Sigma_n) \simeq M(2^k ,\CC)
\end{align*}

With the setup fixed, we will now proceed to define the Dirac operator. Consider $(M,g)$ an oriented Riemannian manifold. A spin structure on $M$ is a principal $\spin(n)$ bundle $P_{\spin(n)}M$ together with a covering map $\pi_2$ making the diagram:
\[
\begin{tikzpicture}
  \matrix (m) [matrix of math nodes,row sep=3em,column sep=5em,minimum width=2em]
  {
      P_{\spin(n)} M & \spin(n) \\
        P_{\so(n)} M & \so(n) \\
	    M & {} \\  
  };
  \path
    (m-1-1) edge node [ above] {$(p,s) \mapsto ps$} (m-1-2)
    (m-2-1)	edge node [below] {$(p,A) \mapsto pA$} (m-2-2);

  \path[->] 
  	(m-1-1) edge node [left]{$\pi_2$} (m-2-1)
  	(m-1-2)	edge node [right] {$\pi_1$} (m-2-2)
  	(m-2-1) edge node [left] {} (m-3-1);
\end{tikzpicture}
\]
commutative, where $\pi_1$ is the standard covering map and $ P_{\so(n)} M$ is the oriented orthonormal frame bundle. If $M$ has a spin structure, we can define the spinor bundle as the associated vector bundle:
\[
S:=P_{\spin(n)}M\times_{cl} \Sigma_n.
\]

\subsection{The Dirac operator}
Take $\nabla$ the Levi-Civita connection on $P_{\so(n)}M$ and lift it to a connection on $P_{\spin(n)}M$. It will induce a connection, also denoted $\nabla$, on the associated vector bundle $S$. Then, the Dirac operator acts on $C^{\infty}(S)$, the space of smooth sections of this bundle, in the following manner:
\begin{align*}
\D:C^{\infty}(S) \longrightarrow C^{\infty} (S) && \D:= cl \circ \nabla.
\end{align*}

Clearly, it is a differential operator of order $1$, but it is well known that it is also elliptic: $\sigma_1 \D (\xi)= cl(\xi)$, and $cl(\xi)^2=-|\xi|^2$. If $M$ is compact, since $\D$ is elliptic and self-adjoint, the classical theory of pseudodifferential operators tells us that its spectrum is real and discrete. This fact remains true also when $M$ is not compact (see section \ref{discreteSpectrumSection}) under certain assumptions which will hold in our setting.

\subsection{Spin structures on hyperbolic surfaces}

From now on we will focus on complete oriented hyperbolic surfaces of finite area. Two models of the universal cover of such a surface will be used: the unit disk 
\begin{align*}
\DD:=\left( \{(x,y)\in \RR^2 : x^2+y^2<1 \}, 
g = \frac{4(dx^2+dy^2)}{(1-r^2)^2} \right)
\end{align*}
and the upper half plane 
\begin{align*}
\HH:=\left( \{(x,y)\in \RR^2 : y>0 \}, g = \frac{dx^2+dy^2}{y^2} \right).
\end{align*}
Recall that the group of oriented isometries of $\HH$ is $\psl :=\slinear_2 (\RR)/ \{\pm 1 \}$ (we denote $ \pi : \slinear_2(\RR) \longrightarrow \psl $ the standard projection) with the action given by:
\begin{align*}
\g z = \frac{az+b}{cz+d}, && \g= \begin{bmatrix} a & b \\
                                                                                  c & d   \end{bmatrix} \in \slinear_2(\RR).
\end{align*}
With the mapping 
\begin{align*}
\g \mapsto \left( \g i, \left\{ \g_{*,i} \left( \frac{\partial}{\partial x} \right) ,\g_{*,i} \left( \frac{\partial}{\partial y} \right) \right\} \right),
\end{align*}
we identify $P_{\so (2)}\HH$ to $ \psl$ as principal $\so(2)$ bundles. We shall identify, from now on, isometries and vector frames. Consider $\G\in \psl$ a discrete subgroup. Then the action:
\begin{align*}
\g_*:P_{\so (2)}\HH\longrightarrow P_{\so (2)}\HH, && 
\g_*\left(z, \{ v_1, v_2 \} \right) = \left( \g z, \{ \g_{*,z}v_1, \g_{*,z}v_2 \} \right),
\end{align*}
where $z\in \HH$, becomes the left multiplication of isometries (via the previous mapping). The quotient by this action is isomorphic to the frame bundle over the quotient $\G \setminus \HH$. 

Moreover, we can identify the spin bundle $P_{\spin (2) }\HH$ with $\slinear_2(\RR)$, the space of $2$ by $2$ matrices of determinant $1$, because $\slinear_2(\RR)$ is the two sheeted universal covering for the group of oriented isometries. We are able to lift the left multiplication by $\Gamma$ to $\slinear _2 (\RR)$ if and only if there exists a spin bundle over the quotient surface. But oriented hyperbolic surfaces always admit spin structures (see, for example \cite{carteMoroianuSpinori}), hence, the short exact sequence:
\begin{align*}
 1 \longrightarrow \{ \pm 1 \} \longrightarrow \tilde{\G}:= \pi^{-1}(\G) \longrightarrow \G \longrightarrow 1,
\end{align*}
admits a right splitting $\rho:\G \longrightarrow \tilde{\G}$, $\pi \circ \rho = \Id_{\G}$. 
If $\G \setminus \HH$ is compact and $\G$ has $2g$ generators then there are in total $2^{2g}$ possible spin structures over $\G \setminus \HH$.

\subsection{Encoding the spin structure in a class function}

From the splitting lemma, since $\{\pm1 \}$ is in the center of $\tilde{\G}$, the existence of $\rho$ implies the existance of a left splitting as well: $\chi:\tilde{\G}\longrightarrow \{ \pm 1 \}$, $\chi \circ \iota = \Id_{\{\pm 1\}}$, where $\iota(-1)=-I_2 \in \gl _2(\RR)$. The way D'Hocker and Phong \cite{dhockerPhong} constructed the $\{ \pm 1 \}$-valued class function (i.e. constant along conjugacy classes) we talked about in the introduction is the following. For each element $\g\in \G$, take $\tilde{\g} \in \slinear_2(\RR)$ the unique lift for which $\tr(\tilde{\g})\geq 2$, and define:
\begin{align*}
\nu : \G \longrightarrow \{ \pm 1 \}, && \nu(\g)=\chi(\tilde{\g}).
\end{align*}
Note that the product of two matrices of positive trace is not necessarily of positive trace. Thus, generally, $\nu$ is not multiplicative. Yet, it is constant along conjugacy classes and satisfies $\nu(\g^n)=\nu^n(\g)$. Clearly, this algebraic definition is closely related to the spin structure. So, it should have a geometric interpretation as well. We introduce an apparently different class function and then prove that it is actually the same as $\nu$. Pick $\g\in \G$. If $\g$ is hyperbolic, then we denote $\eta$ the unique geodesic in $\HH$ (parametrized by $t\in \RR$, with speed $1$) fixed by $\g$. Consider $p_{\eta(t)}$ the orthonormal frame $\{ -J(\dot{\eta}(t)), \dot{\eta}(t) \}$, where $J$ is the standard almost complex structure. Clearly $\g p_{\eta(t)}=p_{\g(\eta(t))}$, for any $t \in \RR$. Thus, if $\tilde{p}\in \slinear_2(\RR)$ is an arbitrary lift of $p$, we can define:
\begin{align}\label{caracter}
\varepsilon : \G \longrightarrow \{ \pm 1 \}, && \rho(\g)\tilde{p}_{\eta(t)} = \varepsilon(\g)\tilde{p}_{\g(\eta(t))}.
\end{align}
From the definition, we can immediately see that $\varepsilon$ is a class function and, moreover, $\varepsilon(\g^n)=\varepsilon^n(\g)$.

The case when $\g$ is parabolic is quite similar. For simplicity, suppose $\g$ is the translation by $1$. Consider $t \mapsto \mu_r(rt)$ the unique horocycle preserved by $\g$ which passes through $ir$, for $r\in (0,\infty)$. Then $p=\{ -J(\dot{\mu_r}), \dot{\mu_r} \}$ is a global orthonormal frame, and thus $\tilde{p}$, its lifting, is constant in $t$ and $r$. Hence, we can define $\varepsilon$ as in $(\ref{caracter})$. Note that this class function carries more geometric insights. For a hyperbolic element we have just proved that $\varepsilon$ is exactly the holonomy of the spin bundle over the quotient $\G\setminus \HH$ along the corresponding closed geodesic. For $\g$ a parabolic, $\varepsilon(\g)$ is the limit of the holonomy along horocycles ``escaping" in the cusp.
\begin{proposition}
The two class functions $\nu$ and $\varepsilon$ coincide.
\end{proposition}
\begin{proof}
Let us suppose that $\g$ is hyperbolic. The other case can be solved similarly. There exists $l>0$ and $a\in \psl$ such that: 
$\g=a e^l a^{-1}$, where $e^l (z) = e^lz$. Then $\eta$, the geodesic of unit speed preserved by $\g$, is parametrized by $\eta(t)=a(ie^t)$. Moreover, with the previous identification, $p_{\eta(t)}=ae^t$, where $p$ is as in the definition of $\varepsilon$. If we fix $A \in \slinear_2(\RR)$ an arbitrarily lift for $a$, then 
$A\begin{bmatrix}
e^{t/2} & 0 \\
0 & e^{-t/2} \\
\end{bmatrix}$ is a lift of $p$. Therefore:
\begin{align*}
\rho(\g) A 
\begin{bmatrix}
1 & 0 \\
0 & 1 \\
\end{bmatrix}  &= \varepsilon(\g) 
A
\begin{bmatrix}
e^{l/2} & 0 \\
0 & e^{-l/2} \\
\end{bmatrix};\\
A^{-1}\rho(\g)\varepsilon(\g)A &=
\begin{bmatrix}
e^{l/2} & 0 \\
0 & e^{-l/2} \\
\end{bmatrix}.
\end{align*}
Since $\tr(\rho(\g)\varepsilon(\g))>2$ and $\chi\circ\rho = \Id$, we get $\nu(\g)=\chi(\rho(\g)\varepsilon(\g))=\varepsilon(\g)$.
\end{proof}
From now on, to avoid a heavy notation, the action of $\G$ on the isomorphisms of the spin bundle given by $\g\mapsto \rho(\g)$ will be denoted by $\g \mapsto \g_*$.

\subsection{Non-trivial spin structures and discrete spectrum} \label{discreteSpectrumSection}
\begin{definition}\label{DefinitionNonTrivial}
Suppose that the finite area quotient $M:=\G \setminus \HH$ has cusps. A spin structure on $M$ is called \emph{non-trivial} if the associated class function has the property $\varepsilon(\g)=-1$ for every primitive parabolic element $\g\in \G$.
\end{definition}

A finite area, complete hyperbolic surface can be compactified by glueing a circle at the ``end" of each cusp. This can be done by attaching a point for each geodesic emerging from the cusp. If the initial surface is endowed with a spin structure this will produce a spin structure on the compactification as well. B\" ar \cite{Bar} showed that if the Dirac operator is invertible on this boundary circles, then its spectrum is discrete.
Later, Moroianu \cite{Moroianu} obtained the discreteness of the spectrum for an arbitrary metric for which the Dirac operator on the boundary at infinity is invertible. He showed that the Dirac operator is fully elliptic in the calculus of cusp pseudo-differential operators constructed by Mazzeo and Melrose \cite{MelrsoseMazzeo} and its resolvent is compact in $L^2$. By considering the holomorphic family of complex powers inside the cusp calculus, he obtained a Weyl law. Note that both \cite{Bar} and \cite{Moroianu} use an analytic definition for non-triviality. Namely a spin structure is \emph{non-trivial} on a cusp if the Dirac operator induced on the circle glued at the end of the cusp is invertible. We will now show that this condition of invertibility is equivalent with Definition \ref{DefinitionNonTrivial}.

\begin{lemma}
The Dirac operator on a finite area hyperbolic surface with cusps has discrete spectrum if and only if the spin structure is non trivial in the sense of Definition \ref{DefinitionNonTrivial}.
\end{lemma}
\begin{proof}
Fix a cusp $C$. In this cusp, the metric can be written as:
\begin{align*}
g=\frac{dx^2}{x^2} + x^2d\theta^2;
&&
(x,\theta)\in \left( 0,\frac{1}{\pi} \right) \times [0,2\pi]=:C.
\end{align*} 
Denote $p:=\{ x\frac{\partial}{\partial x}, \frac{1}{x}\frac{\partial}{\partial \theta} \} \in P_{\so_2(\RR)}C$ an orthonormal on $C$. Then $\varepsilon$ takes value $-1$ on the primitive parabolic element associated to this cusp if and only if $\tilde{p}$, the lift of $p$ along any horocycle to the spin bundle $P_{\spin (2)}C$, takes opposite values at $\theta = 0$ and $ \theta = 2\pi$, and in that case, its rotation by angle $\theta/2$ forms a loop:
\begin{align*}
\tilde{p}_{\vert_{\theta=0}} = - \tilde{p}_{\vert_{\theta=2\pi}};
&&
\left( \tilde{p}e^{i\theta/2}\right)_{\vert_{\theta=0}} = \left(\tilde{p}e^{i\theta/2}\right)_{\vert_{\theta= 2\pi}}.
\end{align*}
It follows that the spinor bundle over the cusp $C$ is spanned by two global sections:
\begin{align*}
\sigma^+:=e^{i\theta/2}[\tilde{p},1];
&&
\sigma^-:=e^{i\theta/2}\left[ \tilde{p}, \frac{e_1-ie_2}{\sqrt{2}} \right].
\end{align*}
By direct computations the Dirac operator is given by the matrix:
\begin{align*}
\D = 
\begin{bmatrix}
0 & \frac{x-1}{2x}+x\frac{\partial}{\partial x}+\frac{i}{x}\frac{\partial}{\partial \theta} \\
-\frac{x+1}{2x}-x\frac{\partial}{\partial x}+\frac{i}{x}\frac{\partial}{\partial \theta} & 0
\end{bmatrix}.
\end{align*}
We glue a $S^1$ at the end of the cusp and endow it with the spin structure induced from the surface. There are only two spin bundles over a circle: either the union of two disjoint circles or a circle winding twice around the base. But $\tilde{p}$ is not a global section in $P_{\spin (2)}C$, hence our induced spin bundle is a circle. Let $q:=\{ \frac{\partial}{\partial\theta }\}$ be the frame given by the oriented unit tangent vector at each point of the circle and $\tilde{q}$ one lift (out of the two possible) to $P_{\spin(1)}S^1$. Then $\sigma := e^{i\theta/2}[\tilde{q}, 1]$ is a global section in the spinor bundle. Moreover, $e^{ik\theta}\sigma$, for $k\in \ZZ$, forms a orthogonal $L^2$ basis for this bundle. But this means that the spectrum of the Dirac operator on $S^1$ is $\ZZ+\frac{1}{2}$ and thus, it is invertible, which is precisely B\" ar's non-triviality condition from \cite[Theorem 2]{Bar}. Similarly, if $\varepsilon$ takes value $1$ on the primitive parabolic element corresponding to $C$, the spectrum is $\ZZ$.
\end{proof}

\subsection{Explicit formulae for the Dirac operator}

For any point $(x,y)\in \HH$ we consider the orthonormal frame $p:=\{ y\frac{\partial}{\partial x}, y\frac{\partial}{\partial y} \}$. If $\tilde{p}$ is one of the two lifts of this frame to $P_{\spin(2)} \HH$, then the spinor bundle is generated by the global sections
\begin{align}\label{sectiuneGlobalaH2}
\sigma^+:=[\tilde{p}, 1] && \sigma^-:=[\tilde{p}, e_1-ie_2].
\end{align}
In the splitting given by these sections, the Dirac operator is given by the off-diagonal symmetric matrix of operators:
\begin{align*}
\D = 
\begin{bmatrix}
0 & -y\frac{\partial}{\partial x} + iy \frac{\partial}{\partial y} -\frac{i}{2} \\
y\frac{\partial}{\partial x} + iy \frac{\partial}{\partial y} -\frac{i}{2} & 0
\end{bmatrix}:
\begin{matrix}
 \ S^+ \\
\oplus \\
 \ S^-
\end{matrix} 
\longrightarrow
\begin{matrix}
 \ S^+ \\
\oplus \\
 \ S^-
\end{matrix}.
\end{align*}
 
In the case of the unit disk $\DD$, we use the orthonormal frame $\{\frac{1-r^2}{2}\frac{\partial}{\partial x}, \frac{1-r^2}{2}\frac{\partial}{\partial y}   \}$, where $r:= \sqrt{x^2+y^2}$.
Using the same sections $\sigma^{\pm}$ as above, the Dirac operator has the form:
\begin{align*}
\D = 
\begin{bmatrix}
0 & - \frac{1-r^2}{2} \frac{\partial}{\partial x} + i\frac{1-r^2}{2} \frac{\partial}{\partial y} -\frac{x-iy}{2} \\
\frac{1-r^2}{2} \frac{\partial}{\partial x} + i\frac{1-r^2}{2}  \frac{\partial}{\partial y} +\frac{x+iy}{2} & 0
\end{bmatrix}.
\end{align*}

We denote $\D^+$ the part of $\D$ mapping $S^+$ into $S^-$ ($\D^-$ is defined similarly). Thus, rewriting the operator in polar coordinates $\frac{\partial}{\partial r}$ and $\frac{\partial}{\partial \theta}$, we get:

\begin{equation}\label{ecuatieDiractPatratRadial}
\begin{split}
\D^- &=  \frac{1}{2} \left( -{re^{-i\theta}} - {e^{-i\theta}(1-r^2)} \frac{\partial}{\partial r} + \frac{ie^{-i\theta}(1-r^2)}{r} \frac{\partial}{ \partial \theta}  \right);  \\
\D^+ &= \frac{1}{2} \left( {re^{i\theta}} + {e^{i\theta}(1-r^2)} \frac{\partial}{\partial r} + \frac{ie^{i\theta}(1-r^2)}{r} \frac{\partial}{ \partial \theta} \right).
\end{split}
\end{equation}

\section{Bounding the number of geodesics on a hyperbolic surface}
In this section we want to bound the number of primitive geodesics on a hyperbolic surface. A lot is known in this direction, we mention Randol \cite{boundGeodesicsRandon}, who proved that on a fixed compact hyperbolic surface,
\[
\pi(x) = \li (x) + O(x^\alpha),
\]
where $\frac{3}{4}\leq \alpha <1 $, $\li (x):= \int_2^{x} \frac{dt}{\log t}$ is the logarithm integral function and $\pi(x)$ represents the number of closed primitive geodesics of length less than or equal to $\log x$. One can prove such a formula by applying the Selberg trace formula for the scalar Laplacian $\Delta$ to a particular class of functions. One problem with this rather exact estimate is that we don't control how the error term behaves when we consider families of hyperbolic metrics (instead of a fixed one). In what follows, we will use the following notations: 
\begin{itemize}
\item $L(r)$, the number of closed geodesics of length at most $r$;
\item $d(p,q)$, the hyperbolic distance between $p$ and $q$;
\item $\diam(M)$, the diameter of the surface $M$;
\item $l(\eta)$, the length of the geodesic $\eta$.
\end{itemize}

\begin{lemma}\label{marginireGeodeziceCompact}
Let $M = \G \setminus \HH $ be a hyperbolic compact surface. Then $L(r) < C e^r$, where $C$ depends only on the diameter and the area of $M$.
\end{lemma}
\begin{proof}
Fix $p$ a point on $M$, consider $\tilde{p}\in \HH$ a lift of $p$ and $D_{\tilde{p}}$ the Dirichlet fundamental domain. Take $\eta$ a closed geodesic on $M$ and fix $q$ a point on it. If $\tilde{q}\in D_{\tilde{p}}$ is the unique lift of $q$ in $D_{\tilde{p}}$, there exists a unique $\g \in \G$ such that the line from $\tilde{q}$ to $\g \tilde{q}$ is the lift of $\eta$. Hence:
\[
l(\eta) = d(\tilde{q}, \g \tilde{q}) > d(\tilde{p}, \g \tilde{p}) - 2\diam(M).
\]
Since we get a different $\g$ for each geodesic $\eta$, it follows that
\[
L(r)< |\{ \g \in \G : d(\tilde{p}, \g\tilde{p}) \leq r+2\diam(M) \}|.
\]
But the right-hand side is clearly bounded by the area of a hyperbolic disk of radius $r+3\diam(M)$ divided by the area of $M$.
\end{proof}

When we vary the metric on $M$, its diameter can explode and in fact we consider precisely deformations of exploding diameter. Hence, we need to refine the above statement. One way to do that is by considering a compact sub-surface which remains bounded throughout the process.
\begin{proposition}\label{marginireGeodeziceParametru}
Let $M= \G \setminus  \HH$ be a compact hyperbolic surface and let $\eta$ be a closed, simple, geodesic. Moreover, let $g_t$ be a family of metrics on $M$ such that $l(\eta_t)$ converges to $0$ as $t$ converges to $0$ (as described in $(\ref{pinchingProcess})$). Then, there exists  a sub-surface $K(t)$ with piecewise smooth boundaries, and a constant $C$ which only depends on the area of $M$ and the diameter of $K(t)$,  such that:
\begin{itemize}
\item $\diam(K(t))$ is uniformly bounded for all $t>0$;
\item $L_{\eta_t}(r) < C e^r$, where $L_{\eta_t}(r)$ counts those closed geodesics, which are \emph{not} multiples of $\eta_t$ (i.e. $t \mapsto \eta_t(nt)$, for any $n$ a positive integer), that are smaller than $r$.
\end{itemize}  
\end{proposition}
\begin{proof}
Consider a maximal system of $3g-3$ simple, closed geodesics on $M$, including $\eta_t$ among them. We cut along these geodesics to decompose the surface into $2g-2$ pairs of pants. Consider a pair whose boundary contains $\eta_t$. Denote $BC$, $DE$ and $FA$ the lines that realize the distance between the boundary components (minimal orthogeodesics), as seen in Figure \ref{pairOfPants}.

\begin{figure}[H]
\centering
\includegraphics[width=4.5cm, height=4.8cm]{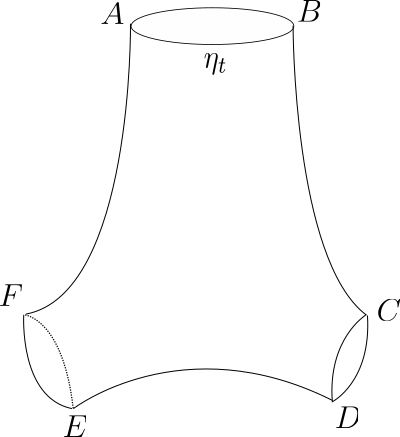}
\caption{Pair of pants}
\label{pairOfPants}
\end{figure}

The points $A$ and $B$ are diametrically opposed on $\eta$ (similarly, $C$ and $D$, $E$ and $F$ are also diametrically opposed). We take $AFEDCD'E'F'A'B$ a fundamental domain such that the geodesic $BC$ corresponds to the imaginary half-line as in Figure \ref{fundamentalDomain}. Continue the line $AF$ until it meets the real axis. From that point we consider the unique geodesic perpendicular to $BC$, and denote $S$ the point of intersection. Construct similarly $R\in AF$ and $R'\in A'F'$.
The union of the line segments $RS$ and $R'S$ will project onto a piecewise smooth loop on $M$. Consider the same loop for the other pair of pants which has $\eta_t$ as boundary (it might be another boundary component of the same pair of pants). These two loops will disconnect the surface $M$ into two connected components: a cylinder containing $\eta_t$ and the rest of the surface, denoted $K(t)$ as seen in Figure \ref{desenK_t}.
We claim that the length of the boundaries of $K(t)$ are bounded from below. Indeed, pick an isometry which maps the point $A$ into $i$ and the point $B$ into $ie^{l(\eta_t)/2}$. A straightforward computation will yield the coordinates of $R$ and $S$:
\begin{align*}
R\left( \frac{2e^{l(\eta_t)/2}}{e^{l(\eta_t)}+1}, \frac{e^{l(\eta_t)}-1}{e^{l(\eta_t)}+1} \right), && S \left( \frac{2e^{l(\eta_t)} }{e^{l(\eta_t)}+1},\frac{e^{l(\eta_t)}(2(e^{l(\eta_t)}-1)}{e^{l(\eta_t)}+1} \right).
\end{align*}
\begin{figure}[H]
\centering
\includegraphics[width=15cm, height=7.5cm]{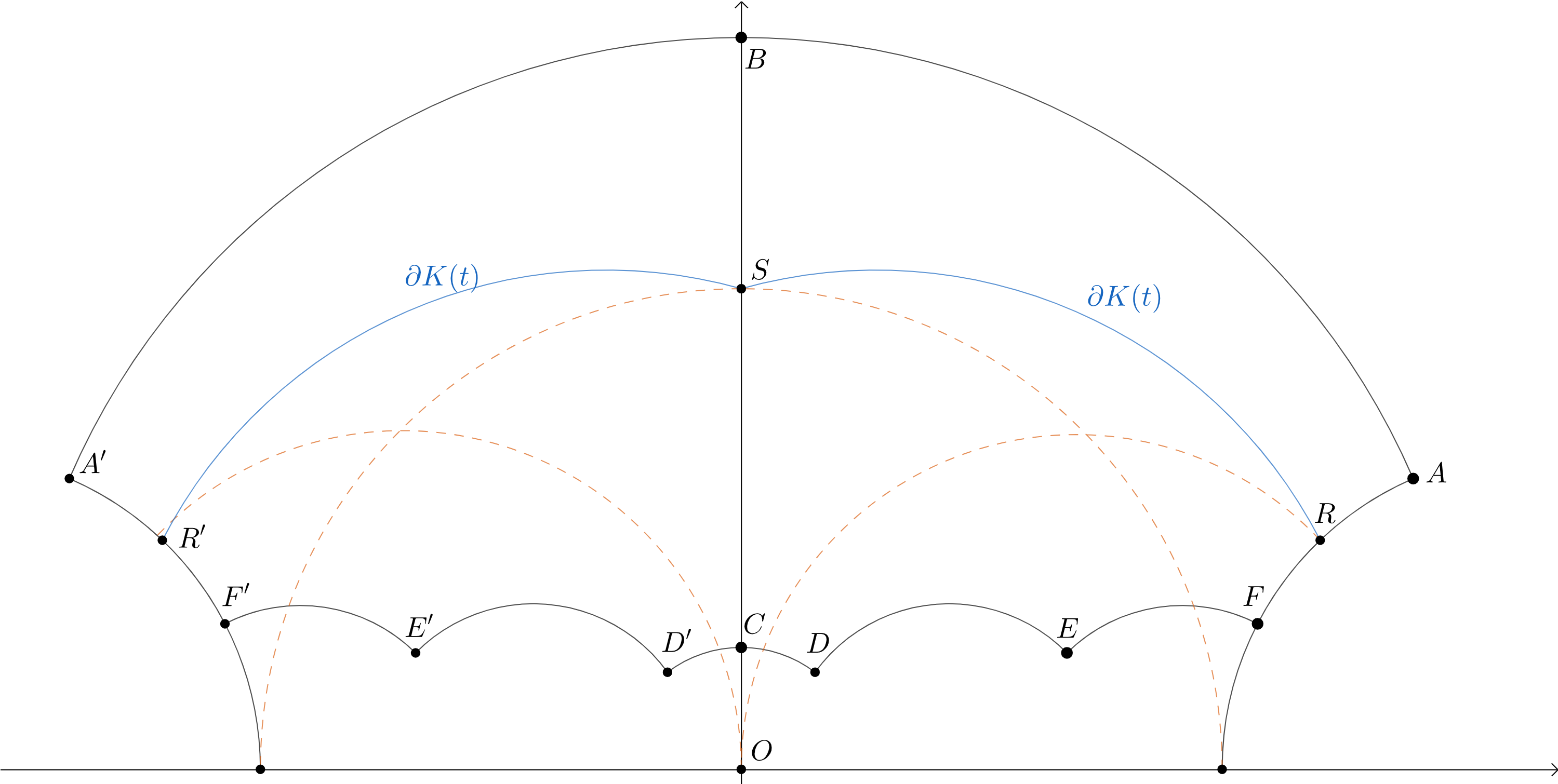}
\caption{A fundamental domain of Figure \ref{pairOfPants}}
\label{fundamentalDomain}
\end{figure}
\noindent
Hence, the distance between these two points satisfies:
\begin{align*}
2\left( \cosh d(R,S) -1 \right)&= \frac{\left( e^{l(\eta_t)}+1 \right)^2}{e^{l(\eta_t)/2}\left( e^{l(\eta_t)/2}+1 \right)^2},
\end{align*}
or, equivalently:
\begin{align*}
\cosh d(R,S)= \frac{2\cosh^2\left( \frac{l(\eta_t)}{4} \right) + \cosh^2 \left( \frac{l(\eta_t)}{2} \right)}{2\cosh ^2 \left( \frac{l(\eta_t)}{4} \right)}.
\end{align*}
Since the expression in the right-hand side decreases to $3/2$ as $l(\eta_t)$ goes to $0$, the distance between $R$ and $S$ remains larger than $\cosh^{-1} (3/2)$. In consequence, $\diam (K(t))$ is uniformly bounded for all $t$.
\begin{figure}[H]
\centering
\includegraphics[width=11.5cm, height=3.5cm]{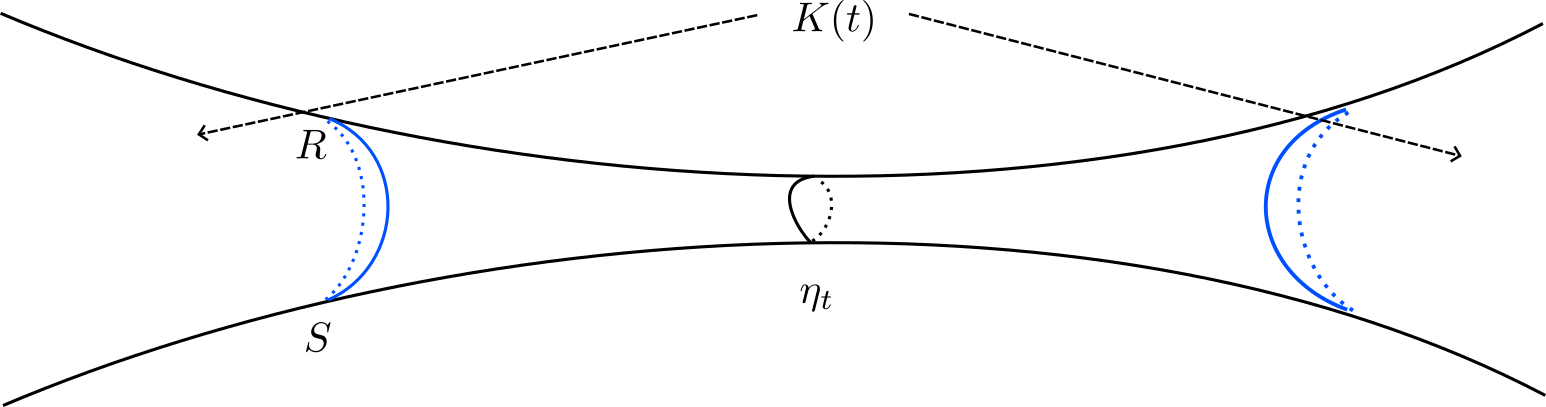}
\caption{The compact region $K(t)$ on the surface}
\label{desenK_t}
\end{figure}
The second part of the proposition can be deduced easily. Take $\eta$ a closed geodesic on $M$, different from $\eta_t$ or its integer powers. Since $M\setminus K(t)$ is topologically equivalent to $\eta_t$, $\eta$ must intersect the interior of $K(t)$, otherwise it would have the same free homotopy class as an integer power of $\eta_t$. Hence, we can follow the reasoning from Lemma \ref{marginireGeodeziceCompact}, the only difference being that instead of an arbitrary point, we take $q\in \eta \cap K(t)$. 
\end{proof}

\begin{lemma}\label{marginireGeodeziceNonCompact}
Let $M= \G \setminus \HH$ be a complete hyperbolic surface of finite area (with cusps). Then $M$ can be decomposed in a compact subsurface $K$ and a finite number of cusps, by cutting along horocycles of length $2$, exactly one for each cusp. Moreover, the infinite part of each cusp is isometric to the quotient of $\{ z\in \HH : 1<\Im (z) \text{ and } 0\leq \Re (z) < 2\}$ by $z\mapsto z+2$. Therefore $L(r) < Ce^r$, where $C$ depends only on $\diam(K)$ and the area of $M$.
\end{lemma}
\begin{proof}
For simplicity we suppose the surface has only one cusp. Consider a decomposition into pairs of pants. The pair, say $P$, that contains the cusp will have a boundary of length $0$ and two other boundaries of positive length. Denote $\infty A$, $BC$ and $F\infty$ the lines that realize the distance between the boundary components, as in Figure \ref{desen_pantalon_cusp}.
\begin{figure}[H]
\centering
\includegraphics[width=10cm, height=3.5cm]{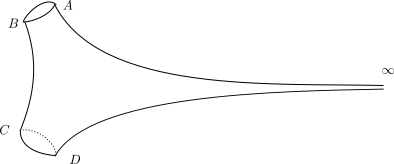}
\caption{The pair of pants $P$}
\label{desen_pantalon_cusp}
\end{figure}
\noindent
Cut the pair of pants along the orthogeodesics $\infty A$, $BC$ and $D\infty$ and consider one of the two resulting pentagons (degenerated hexagons). Moreover, take a lift of this pentagon say, $\tilde{A}\tilde{B}\tilde{C}\tilde{D}\infty$, into $\HH$ such that the line $\infty \tilde{A}$ is $\{ \Re(z)=0 \}$ and the line $\infty \tilde{D}$ is $\{ \Re(z)=1 \}$. Notice that $AB$ and $CD$ are both perpendicular on $BC$ thus the lines $\tilde{A}\tilde{B}$ and $\tilde{C}\tilde{D}$ must not intersect. But this happens only if the sum of the euclidean radii of the two half circles $\tilde{A}\tilde{B}$ and $\tilde{C}\tilde{D}$ is smaller than $1$. Hence, both $\tilde{A}$ and $\tilde{B}$ lie below the horocycle $\{ \Im(z)=1: \Re(z)\in[0,1] \}$. Making the same construction for the other pentagon we obtain a horocycle of length $2$ on our initial surface, delimiting a compact region, denoted $K$, from the cusp. As in the proof of Proposition \ref{marginireGeodeziceParametru} any closed geodesic must intersect $K$, thus we finish the proof by following the argument from Lemma \ref{marginireGeodeziceCompact}.
\end{proof}


\section{Trace formula on hyperbolic surfaces}

The trace formula was first introduced by Selberg \cite{SelbergOriginal}, in 1957. In \cite{BolteStiepanSelbergForDirac}, Bolte and Stiepan prove the formula for the Dirac operator on a compact surface. In this section we give a self-contained elementary proof of the trace formula for the Dirac operator on complete oriented hyperbolic surfaces of finite area. We treat simultaneously both the compact and the non-compact case, paying more attention to the latter. This trace formula is not really new, in \cite{Warner}, Warner obtained a trace formula for arbitrary bundles over finite-volume rank one locally symmetric spaces. There are other authors who restrict their studies to special cases, e.g. in \cite{Hoffmann}, Hoffmann obtains a trace formula for the universal covering group of $\slinear_2(\RR)$. From these results, Theorem \ref{selbergTraceFormula} would follow by translating the languages and observing that most ingredients in the general formula vanish. However, the general proof is complicated and long, mostly because one has to deal with the possible continuous spectrum. Our proof is short and, we hope, easier to read.

Note that $\D^+ \D^-$ and $\D^- \D^+$ have the same spectrum on surfaces (in higher dimensions $0$ may have different multiplicities in the two spectra). So we shall only present the case of positive spinors, i.e. we will look at the operator $\D^-\D^+$. 

\begin{definition}\label{definitiePrimitiv}
Let $\G$ be a discrete subgroup of $\psl$. We say that $\mu\in\G$ is \emph{primitive} if it cannot be written as $\mu=\g^n$, for $n\geq 2$ and $\g\in\G$.
\end{definition}
\begin{remark*}
On a hyperbolic surface $M = \G \setminus \HH$, the set of closed, oriented geodesics is in one to one correspondence with the hyperbolic conjugacy classes of $\G$. Hence, by $l(\g)$ we denote the length of the unique geodesic associated to the class $[\g]\subset \G$.
\end{remark*}

\begin{theorem}\label{selbergTraceFormula}
Let $M=\G \setminus \HH$ be a finite area hyperbolic surface, with $k$ cusps. Suppose that $\varepsilon(\g)=-1$ for each primitive parabolic element $\g\in\G$. Consider $\{\lambda_j\}_{j\in {\mathbb{N}}}$ the eigenvalues of $\D^{-}\D^+$ on $M$ and choose a sequence $\{\xi_j \}_{j\in {\mathbb{N}}} \subset \CC$ such that $\lambda_j=\xi_j^2$. Let $\phi \in C^{\infty}_c(\RR)$ and define $a,v,u:\RR \longrightarrow \RR$ by:
\begin{align*}
& a(t)=\frac{\phi(t)}{\sqrt{t+4}};\\
&v(t)=4\cosh\left(\tfrac{t}{2} \right)\int_{0}^{\infty}a\left( 4 \sinh^2 \left(\tfrac{t}{2} \right)+y^2 \right)dy ; \\
&u(\xi)=\hat{v}(\xi)=\int_{\RR}v(t)e^{-i\xi t}dt.
\end{align*}
Then:
\begin{align*}
\sum_{j=0}^{\infty} u(\xi_j)=\frac{\area(M)}{4\pi}\int_{\RR} \xi u(\xi)\coth(\pi \xi)d\xi+\sum_{[\mu] } \sum_{n=1}^{\infty} \frac{l(\mu)\varepsilon^n(\mu)v(nl(\mu))}{2\sinh \left(\frac{nl(\mu)}{2}\right)}-k\log(2)v(0),
\end{align*}
where $[\mu]$ runs over all conjugacy classes of primitive, hyperbolic elements in $\G$.
\end{theorem}

\subsection{Eigenspinors of Dirac on the hyperbolic plane}

Let $\tilde{S}$ be the spinor bundle on the universal covering of $M$. Consider the following smoothing operator:
\begin{align*}
\Phi : C^{\infty}(\mathbb D, \tilde{S}) \longrightarrow C^{\infty}(\mathbb D, \tilde{S}) &&
\Phi s (z) := \int_{\DD} \phi (2\cosh d(z,z')-2)\tau_{z'\mapsto z} s(z')dg_{\DD }(z'),\\
\end{align*}
where $\tau_{z'\mapsto z}$ is the parallel transport in $\tilde{S}$ with respect to $\nabla$ from $z'$ to $z \in \mathbb D$. Note that if $\sigma(z)$ is a positive spinor, then so is $\tau_{z'\mapsto z} \sigma(z')$. We want to show that $\Phi$ and $\D^-\D^+$ have the same eigenspinors. In order to do this, we need some preparations. Define $\PP$ to be the projector:
\begin{align*}
\PP : C^{\infty}(\tilde{S}) \longrightarrow C^{\infty}(\tilde{S}) && 
\PP \sigma(z) = \PP f\sigma^+ (z) :=\int_0^{1} f(ze^{2\pi it})dt \ \sigma^+(z),
\end{align*}
where $\sigma^+$ was defined in $(\ref{sectiuneGlobalaH2})$. We have three immediate properties:
\begin{itemize}
\item[1)] $\PP ^2 =P$;
\item[2)] $\PP \circ \F (0)=\F \circ \PP (0)$;
\item[3)] $\PP \circ \D^-\D^+ =\D^-\D^+\circ \PP$.
\end{itemize}
The first one is obvious. To see the second property, we parametrize the unit disk $(r,\theta) \mapsto z$ by polar coordinates, $(r,\theta) \in [0,1)\times [0,2\pi]$. Notice that $\sigma^+$ is parallel along the radii of the disk, hence $\tau_{z'\mapsto 0} f\sigma^+(z')=f(z')\sigma^+(0)$. The third identity follows from the fact that $\PP$ commutes with both $\frac{\partial
}{\partial r}$ and $\frac{\partial}{\partial \theta}$ (see $(\ref{ecuatieDiractPatratRadial})$).

Therefore, if $f\sigma^+$ is an eigenspinor for $\D^-\D^+$, then so is $\PP f\sigma^+$. Moreover, combining the two formulae from $(\ref{ecuatieDiractPatratRadial})$, $h=\PP f$ is a solution of the second order differential equation:
\begin{align}\label{ecuatiODE2}
\left( -\frac{2-r^2}{4} - \frac{(1-r^2)^2}{4r} \frac{\partial}{\partial r} - \frac{(1-r^2)^2}{4}\frac{\partial^2}{\partial r ^2} \right) h = \lambda h,
\end{align}
where $h(r)$ is smooth on $[0,1)$. If we denote $h'$ by $\tilde{h}$, then $(\ref{ecuatiODE2})$ is equivalent to the system:
\begin{align*}
\renewcommand\arraystretch{1.2} 
\begin{bmatrix}
0 & 1\\
-\frac{2 +4\lambda -r^2}{(1-r^2)^2} & -\frac{1}{r} \\
\end{bmatrix}
\begin{bmatrix}
h\\
\tilde{h}
\end{bmatrix} = 
\begin{bmatrix}
h' \\
\tilde{h}' 
\end{bmatrix}.
\end{align*}
Therefore, the Wronskian is a solution of:
\[
w' + \frac{w}{r} =0,
\]
with the general solution:
\[
w=\frac{1}{r} \cdot c, 
\]
where $c$ is a constant.
\begin{lemma} \label{spaceDimension}
The space of solutions of $(\ref{ecuatiODE2})$ which are smooth at $r=0$ has dimension $1$.
\end{lemma}
\begin{proof}
Since the Wronskian blows up at $0$, the dimension is at most $1$. Thus, it is enough to find one smooth solution. A direct computation tells us that the section $x+iy \mapsto y^{\xi}s^+$ is an eigenspinor of $\D^-\D^+$ of eigenvalue $-(\frac{1}{2}-\xi)^2$ on the upper half-plane $\mathbb H$. Pulling this section back on $\DD$ and applying $\PP$ finishes the lemma.
\end{proof}
We are now ready to prove the following proposition:
\begin{proposition}\label{propositionSameEigenspinors}
For any $\xi \in \CC$ there exists $u(\xi)\in \CC$ such that if\begin{align*}
f:\DD \longrightarrow \CC &&
\D^-\D^+ (f\sigma^+) = \xi^2 f\sigma^+,
\end{align*}
then
\[
\F(f\sigma^+) = u(\xi) f\sigma^+,
\]
for the function $u$ appearing in the statement of Theorem $\ref{selbergTraceFormula}$.
\end{proposition}
\begin{proof}
First, we claim that is enough to prove the equality at $0$. Notice that $\g_*$ preserves Clifford multiplication. Moreover, since it is an isometry, $(\g^{-1})^*$ preserves the connection $\nabla$ and thus, the parallel transport. Combining these facts we get that
\begin{align}\label{eigenspinoriPhi}
\D^-\D^+\g_*(f\sigma^+) = \g_*\D^-\D^+f\sigma^+;
&& \F f\sigma^+ (z) = \g_*\left( \F \g^{-1}_*(f\sigma^+) (0) \right).
\end{align} 
We will now see why the equality occurs at the origin. Note that, from the previous lemma, $\PP f\sigma^+$ is completely determined by its value at $0$. Thus:
\begin{align*}
\F f\sigma^+(0) &= \PP \F f\sigma^+ (0) = \F \PP f\sigma^+ (0) 
=: u(\xi)f\sigma^+ (0).
\end{align*}
An elementary computation tells us what the parallel transport is:
\begin{align}\label{calculTransportParalel}
\tau_{z\mapsto w} = -i \frac{z-\overline{w}}{|z - \overline{w}|}.
\end{align}
With this in mind, for the final part of the proposition, we move everything on the upper half plane $\HH$ and carry out the computation there. It is enough to compute the value of $\F$ on a fixed eigenfunction, specifically $x+iy\mapsto y^{-i\xi + 1/2}\sigma^+$: 
\begin{align*}
\F \left( y^{-i\xi+1/2} \sigma^+ \right)(i) &= \int_{\HH} \phi \left( 2\cosh d(i, z') -2 \right) \tau_{z' \mapsto i} \left(  \Im (z')^{-i\xi+1/2}\sigma^+ \right)  dg_{\HH}(z') \\
&=\sigma^+(i) \int_0^{\infty}\int_{\RR} \phi \left( \frac{x^2+(y-1)^2}{y} \right) y^{-3/2-i\xi}\frac{1+y}{\sqrt{x^2 + (y+1)^2}}  dxdy\\
&=\sigma^+(i) \int_0^{\infty}\int_{\RR} \phi \left( X^2 + \frac{(y-1)^2}{y} \right) y^{-1-i\xi} \frac{\frac{1+y}{\sqrt{y}}}{\sqrt{X^2+\frac{(y+1)^2}{y}}}dXdy \\
&=\sigma^+(i) \int_{\RR}\int_{\RR} \phi \left( X^2 + 4\sinh^2\left( \frac{t}{2} \right) \right) e^{-i\xi t} \frac{2\cosh \left( \frac{t}{2} \right)}{\sqrt{X^2 + 4 \cosh^2 \left( \frac{t}{2} \right)}} dX dt \\
& = \sigma^+(i)\hat{v}(\xi).
\end{align*}
Throughout the computation we have: 
\begin{itemize}
\item changed variables from $\{ x,y \}$ to $\{ X :=\frac{x}{\sqrt y}, y \}$;
\item changed variables $y = e^t$, $t \in \mathbb R$. \qedhere
\end{itemize} 
\end{proof}

\subsection{A pretrace formula}
We want Proposition \ref{propositionSameEigenspinors} to hold true for an operator acting on spinors on $M$ as well. Thus let us define the kernel:
\begin{align}\label{kernelG}
G(z,z'):= \sum_{\g \in \G}\phi(2\cosh d(\g^{-1}z,z')-2)\g_* \tau_{z'\mapsto \g^{-1}z},
\end{align}
for $z,z'\in \HH$. It is locally finite, since $\phi$ has compact support. Note that the $G$ descends to $M\times M$, since for every $\alpha, \beta \in \G$:
\[
G(\alpha z, \beta z') = \alpha_* G(z,z')\beta^{-1}_*.
\]
Hence, it produces an operator:
\begin{align*}
\mathcal{G}:C^{\infty}(M,S) \longrightarrow C^{\infty}(M,S); && 
\mathcal{G}\sigma(z) = \int_M G(z,z')\sigma(z') dg(z').
\end{align*}
If $M$ is compact, the sum becomes globally finite. Otherwise, we claim that it is bounded.
\begin{proposition}\label{marginireKernelG}
The sum defining the kernel $G$ is uniformly bounded on $M\times M$.
\end{proposition}
\begin{proof}
Let $F$ be a Dirichlet domain for $M$. We will show that the sum defining $G$ is uniformly bounded on $F\times F$. Take $r$ a positive real number large enough such that $\phi(2\cosh d(z,z')-2) = 0$ whenever $d(z,z')\geq r$. From Lemma \ref{marginireGeodeziceNonCompact} there exist unique horocycles of length $2$ delimiting a compact region of $K\subset M$ from the cusps. Take $\tilde{K}$ the lift of $K$ in $F$. For $\delta > 0$, denote by $K+\delta:=\{ p\in M : d(p,K)\leq \delta \}$ and by $\tilde{K}+ \delta:= \{ w\in \HH : d(w,K)\leq \delta \}$.

If $\pi(z') \in K+r$, then $z \mapsto \phi(2\cosh d(\g^{-1}z,z')-2)$ vanishes whenever $\g^{-1}F$ does not intersect the compact set $\tilde{K}+2r$. We claim that this compact set intersects at most a finite number (say $N$) of fundamental domains $\g^{-1}F$. Indeed, suppose there exists a sequence $(\g_n)_{n\in \NN}\subset \G$ such that $\g^{-1}_n F\cap \tilde{K}+2r \neq \emptyset$. There exists a sequence $(x_n)_{n\in \NN}\subset F$ such that $\g^{-1}_n F\cap \tilde{K}+2r \ni \g^{-1}_nx_n $ converges to $y\in \tilde{K}+2r$. Taking the projection onto $M$, we get that $K+2r \ni \pi(x_n)$ converges to $\pi(y) \in K+2r$. Moreover, since $\pi(x_n) \in K+2r$, it follows that $x_n \in F\cap\tilde{K}+2r$, thus $x_n$ converges to a point $x\in F\cap\tilde{K}+2r$. Notice that $x$ and $y$ have the same projection onto $M$, hence there exists $\g\in \G$ such that $\g y=x$, which in turn implies that $\g\g_n^{-1}x_n$ converges to $x$. But recall that $x_n$ also converges to $x$, and this contradicts the fact that $\G$ acts properly discontinuous. Therefore for every $z\in F$ and every $z'\in F$ with $\pi(z')\in K+r$, at most $N$ terms in the sum $(\ref{kernelG})$ do not vanish.

If $\pi(z')\notin K+r$, then $\pi(z')$ lies in a cusp and the distance between itself and the unique horocycle associated to the said cusp is at least $r$. Furthermore, we have shown in Lemma \ref{marginireGeodeziceNonCompact} that this cusp is embedded in a part of the vertical strip of width $1$. Hence $z \mapsto \phi(2\cosh d(\g^{-1}z, z')-2)$ does not vanish only if $z$ is above the horocycle and $\g$ is a translation. Thus, writing $z=x+iy$ and $z'=x'+iy'$, $G$ reads:
\begin{align*}
G(z,z')& =  \sum_{n \in \ZZ }\phi(2\cosh d(z-n,z')-2)(-1)^n\tau_{z'\mapsto z-n} \\
& = \sum_{n \in \ZZ } a\left( \frac{(x-n-x')^2 + (y-y')^2}{yy'} \right) e^{i\pi n} \frac{y+y' + i(x-n-x')}{\sqrt{yy'}},
\end{align*}
where $a$ is the function defined in Theorem \ref{selbergTraceFormula}. A fruitful way to regard this sum is by the Poisson summation formula:
\begin{align}\label{kernelGAfterPoisson}
G(z,z')& = \sum_{m\in \ZZ} \int_{\RR}a\left( \frac{(x-n-x')^2 + (y-y')^2}{yy'} \right) e^{i\pi n(1-2m)} \frac{y+y' + i(x-n-x')}{\sqrt{yy'}} dn.
\end{align}
Next, we take each term, denoted $F(m,y,y')$, separately and change the variable from $n$ to $X:=(n+x-x')(yy')^{-1/2}$:
\begin{align*}
F(m,y,y') = e^{i\pi(x-x')(1-2m)} \int_{\RR} a \left( X^2+\frac{(y-y')^2}{yy'} \right) e^{i\pi X \sqrt{yy'}(1-2m)} (y+y'-iX\sqrt{yy'})dX.
\end{align*}
Since $y$, $y'$ and $1-2m$ do not vanish, we claim that, for any $p$ a positive integer, $|\sqrt{yy'}(1-2m)|^p F(m,y,y')$ goes to $ 0$ as $|\sqrt{yy'}(1-2m)|$ goes to $\infty$. This result is a particular case of the stationary phase principle. Yet, it can be easily seen by integrating by parts: 
\begin{align*}
F(m,y,y')&=\frac{e^{i\pi(x-x')(1-2m)}}{i\pi \sqrt{yy'}(1-2m)} \cdot  \\
&\cdot \int_{\RR} a \left( X^2+\frac{(y-y')^2}{yy'} \right) \frac{\partial }{\partial X}\left( e^{i\pi X \sqrt{yy'}(1-2m)} \right) (y+y'-iX\sqrt{yy'})dX \\
&=-\frac{e^{i\pi(x-x')(1-2m)}}{i\pi \sqrt{yy'}(1-2m)} \cdot  \\   & \cdot \int_{\RR}\frac{\partial}{\partial X}\left( a \left( X^2+\frac{(y-y')^2}{yy'} \right) (y+y'-iX\sqrt{yy'})\right) e^{i\pi X \sqrt{yy'}(1-2m)} dX.
\end{align*}
Continue inductively $p$ more steps, and then notice that the integral is bounded because $a$ has compact support. Therefore, the sum $(\ref{kernelGAfterPoisson})$ is uniformly bounded in the cusp as well.
\end{proof}

\begin{remark*} The hypothesis $\varepsilon(\g) = -1$, for each $\g$ a parabolic, primitive isometry, means that the power of the exponential in the integral is a multiple of $1-2m$, which does not vanish since $m$ is an integer. This fact was essential in our proof.
\end{remark*}

Consider $\sigma_j \in C^\infty(M)$ an eigenspinor of $\D^-\D^+$ with eigenvalue $\xi^2_j$. Since the following diagram commutes: 
\[
\begin{tikzpicture}
  \matrix (m) [matrix of math nodes,row sep=3em,column sep=4em,minimum width=2em]
  {
      P_{\spin(n)}\HH \times_{\rho} \Sigma_2  & P_{\spin(n)}M\times_{\rho} \Sigma_2 \\
      \HH  &   M   \\
  };
  \path[-stealth]
    (m-1-1) edge node [] {} (m-1-2)
    (m-1-2)	edge node [] {} (m-2-2)
    (m-2-1)	edge node [] {} (m-2-2)
    (m-1-1) edge node [] {} (m-2-1);
\end{tikzpicture}
\]
we can take $\tilde{\sigma}_j$ a lift of $\sigma_j$. Using Proposition \ref{propositionSameEigenspinors} and changing variables we get:
\begin{align*}
u(\xi_j)\tilde{\sigma}_j (z) &= \F \tilde{\sigma}_j(z) = \sum_{\g \in \G} \int_{F}   \phi(2\cosh d(\g^{-1}z,z')-2)\g_* \tau_{z'\mapsto \g^{-1}z}\tilde{\sigma}_j(z') dg_{\HH}(z')\\
&= \int_F G(z,z')\tilde{\sigma}_j(z') dg_{\HH}(z') = \mathcal{G}\sigma_j(z),
\end{align*}
where $F$ is a fundamental domain of our surface $M$. Proposition \ref{marginireKernelG} allowed us to commute the integral with the sum. This means that the operator $\mathcal{G}$ has the same eigenspinors and the same eigenfunctions as the operator $\F$. Next we prove an adapted version of Mercer's theorem:
\begin{proposition}\label{traceClass}
\[
\tr \mathcal{G} = \int_{M}G(z,z)dg_M(z).
\]
\end{proposition}
\begin{proof}
Notice that since $G$ is bounded on $M\times M$, it follows that $\mathcal{G}$ acts on $L^2$ sections. We claim that $(1+\D^-\D^+)\mathcal{G}$ is a Hilbert-Schmidt operator. Indeed, computing the square of its Hilbert-Schmidt norm for the basis $\{ \sigma_j \}_{j\in \NN}$ we get:
\[
\| (1+\D^-\D^+) \mathcal{G} \|^2_{HS} = \sum_{j=0}^{\infty} \| (1+\D^-\D^+) \mathcal{G} \sigma_j \|^2_{L^2} = \sum_{j=0}^{\infty} (1+\xi_j^2)^2u^2(\xi_j) <\infty.
\]
The last inequality holds because $u$ is Schwartz and the spectrum of the Dirac operator satisfies $\xi_j \sim j^{1/2}$, a generalised version of Weyl's law (see \cite[Theorem 3]{Moroianu}) on open manifolds. Similarly, one can easily see that $(1+\D^-\D^+)^{-1}$ is Hilbert-Schmidt as well. Thus $\mathcal{G}$ is of trace class and the conclusion follows.
\end{proof}

Thus, starting from Proposition \ref{traceClass} we can rewrite the integral along the diagonal as follows:

\begin{align}\label{pretraceFormula}
\begin{split}
\sum_{j=0}^{\infty} u(\xi_j) &= \int_{M}G(z,z)dg_M(z) = \int_{F}\sum_{\g \in \G}\phi(2\cosh d(\g^{-1}z,z)-2)\g_* \tau_{z\mapsto \g^{-1}z} dg_{\HH}(z) \\
&= \phi(0)\area(M) + \sum_{\g \neq 1}\int_{F}\phi(2\cosh d(\g^{-1}z,z)-2)\g_* \tau_{z\mapsto \g^{-1}z} dg_{\HH}(z).
\end{split}
\end{align}

\subsection{Proof of Theorem \ref{selbergTraceFormula}}
We start with the relation $(\ref{pretraceFormula})$ and compute the right-hand side by rearranging the terms in conjugacy classes, following Selberg's original idea.

\begin{lemma}\label{calculPhi0}
For functions $\phi$ and $u$ as in Theorem \ref{selbergTraceFormula} we have:
\[
\phi(0) = \frac{1}{4\pi}\int_{\RR} \xi u(\xi) \coth(\pi \xi) d\xi
\]
\end{lemma}
\begin{proof}
We start with the function $a$ from Theorem \ref{selbergTraceFormula}. Using polar coordinates we see that:
\[
-\frac{\pi}{4}a(z) = \int_0^{\infty}\int_0^{\infty} a'(z+x^2+y^2) dx dy.
\]
Changing the variable $y=2\sinh\left( \frac{t}{2} \right) $ and evaluating at $0$ we get:
\begin{align*}
a(0) & = -\frac{4}{\pi} \int_0^{\infty}\int_0^{\infty} \frac{\partial}{\partial t} a \left( x^2+ 4\sinh^2 \left(\frac{t}{2} \right)\right) \frac{1}{4\sinh \left( \frac{t}{2} \right)} dxdt \\
&= -\frac{1}{2\pi} \left( \int_0^{\infty} \frac{v'(t)}{\sinh(t)}dt + \int_0^{\infty} \frac{v(t)}{4\cosh ^2\left( \frac{t}{2} \right)} dt \right) .
\end{align*}

Now we proceed to compute each integral separately. Starting with the first one, we apply the Fourier inverse formula for the derivative of the even function $v$ and write $\sinh(t)$ as a power series
\begin{align*}
\int_0^{\infty}\frac{v'(t)-v'(-t)}{2\sinh(t)}dt &= \frac{1}{2\pi} \int_0^{\infty} \int_{\RR} \sum_{m=0}^{\infty} u(\xi)i\xi e^{-t(2m+1)}\left( e^{i\xi t} - e^{-i\xi t} \right)d\xi dt \\
&= \frac{1}{4}\int_{\RR}\xi u(\xi) \sum_{m=0}^{\infty} \frac{i}{\pi} \left( \frac{2}{1+2m -i\xi} - \frac{2}{1+2m + i\xi} \right) d\xi \\
& = -\frac{1}{4}\int_{\RR} \xi u(\xi) \tanh \left(\frac{\pi \xi}{2} \right) d\xi.
\end{align*}
The last equality can be seen by applying the identity:
\begin{align}\label{identitateTanh}
\frac{i}{\pi} \sum_{m=0}^{\infty} \left( \frac{1}{\frac{1}{2}+m -iz} - \frac{1}{\frac{1}{2}+m + iz} \right) = - \tanh (\pi z),
\end{align}
for $\xi = 2z$.  The identity $\eqref{identitateTanh}$ is well-known, one can prove it using the Poisson summation formula for $x\mapsto \frac{1}{\pi}\frac{\xi}{(x+1/2)^2 + \xi^2}$, followed by the residue theorem (see e.g. \cite{MarklofSelbergTraceFormulaIntroduction}). For the second integral, we use again the fact that $v$ is even and apply the Fourier inversion formula. Next, we compute the integral in the variable $t$ with the residue theorem
\begin{align*}
\int_{\RR} \frac{v(t)}{8\cosh^2\left( \frac{t}{2} \right)} dt &=\frac{1}{16\pi} \int_{\RR}\int_{\RR} \frac{u(\xi)e^{i\xi t}}{\cosh^2\left(\frac{t}{2}\right)}dt d\xi = \frac{1}{4\pi}\int_{\RR}\frac{u(\xi)\xi}{\sinh(\pi \xi)}. \qedhere
\end{align*}
\end{proof}

We shall now continue with the studying of the second term in the right-hand side of $(\ref{pretraceFormula})$. Two classical lemmas will play an important part.
\begin{lemma}\label{ciclicInfinit}
Let $\g\in \G$ be a non-trivial element (either hyperbolic or parabolic). Then the centralizer subgroup $C(\g)$ is infinite cyclic.
\end{lemma}

\begin{lemma}
Let $\g \in \G$ and $\mu$ be a primitive (in the sense of \ref{definitiePrimitiv}) element in $C(\g)$(it exists from the previous lemma). Define $A_{\mu}$ to be a system of representatives for $\G / \langle\mu \rangle$ (not necessarily a subgroup). Then the conjugacy class $[\g]$ can be described as $\{ \alpha \g \alpha^{-1}: \alpha \in A_{\mu} \}$. 
\end{lemma}
We can rewrite the second term in the right-hand side of $(\ref{pretraceFormula})$ as:
\begin{align*}
& \sum_{\g \neq 1}\int_{F}  \phi(2\cosh d(\g^{-1} z,z)-2)\g_* \tau_{z\mapsto \g^{-1}z} dg_{\HH}(z) \\
&= \sum_{[\mu]} \sum_{n=1}^{\infty}\sum_{\alpha \in A_{\mu}} \int_{F}  \phi(2\cosh d(\alpha \mu^{-n} \alpha^{-1}z,z)-2)(\alpha \mu^n \alpha^{-1})_* \tau_{z\mapsto (\alpha \mu^{-n} \alpha^{-1})z} dg_{\HH}(z).
\end{align*}

We only need one last result to complete the proof. 
\begin{lemma}\label{calculSumaAlpha}
The sum over $\alpha$ can be computed in the following way:
\begin{itemize}
\item If $\mu$ is hyperbolic, then:
\begin{align*}
\sum_{\alpha \in A_{\mu}} \int_{F}  \phi(2\cosh d(\alpha \mu^{-n}& \alpha^{-1}z,z)-2) \\
& \cdot(\alpha \mu^n \alpha^{-1})_* \tau_{z\mapsto (\alpha \mu^{-n} \alpha^{-1})z} dg_{\HH} = \frac{l(\mu) \varepsilon^n(\mu) v(nl(\mu)) }{2\sinh \left( \frac{nl(\mu)}{2}\right) };
\end{align*}
\item If $\mu$ is parabolic, then:
\begin{align*}
&\sum_{\alpha \in A_{\mu}} \int_{F}  \phi(2\cosh d(\alpha \mu^{-n} \alpha^{-1}z,z)-2)(\alpha \mu^n \alpha^{-1})_* \tau_{z\mapsto (\alpha \mu^{-n} \alpha^{-1})z} dg_{\HH} \\
&+\sum_{\alpha \in A_{\mu^{-1}}} \int_{F}  \phi(2\cosh d(\alpha \mu^{n} \alpha^{-1}z,z)-2)(\alpha \mu^{-n} \alpha^{-1})_* \tau_{z\mapsto (\alpha \mu^{n} \alpha^{-1})z} dg_{\HH} \\
&= \frac{\varepsilon^n(\mu)v(0)}{n}.
\end{align*}
\end{itemize}
\end{lemma}
\begin{proof}
We start with the first part. By changing variables in each term of the sum $w = \alpha^{-1}z$ we obtain an integral on $\cup_{\alpha \in A_{\mu}} \alpha^{-1}F$. This set is in fact a fundamental domain for the quotient $\langle\mu \rangle \backslash \HH$. Since the integrand descends on this cylinder, we can compute it on a more convenient fundamental domain, namely the horizontal band $C:=\{w\in \HH : 1 \leq \Im w <e^l  \}$, where $l$ is the length of the unique geodesic associated to $\mu$ (i.e. $l=2\cosh^{-1}|\tr(\mu)/2|$). We will also use the definition of the class function $\varepsilon$ defining the spin structure $(\ref{caracter})$, and the formula for the parallel transport $(\ref{calculTransportParalel})$. Thus, our initial sum becomes:
\begin{align*}
& \sum_{\alpha \in A_{\mu}} \int_{F}  \phi(2\cosh d(\alpha \mu^{-n} \alpha^{-1}z,z)-2)(\alpha \mu^n \alpha^{-1})_* \tau_{z\mapsto (\alpha \mu^{-n} \alpha^{-1})z} dg_{\HH}(z) \\
&= \int_{\langle\mu \rangle \backslash \HH} \phi(2\cosh d( \mu^{-n} w,w)-2)( \mu^n)_* \tau_{w \mapsto  \mu^{-n} w} dg_{\langle\mu \rangle \backslash\HH}(w) \\
&= \int_{C} \phi(2\cosh d( w,e^{nl} w)-2) \varepsilon^n(\mu) (-i) \frac{w-\overline{w}e^{nl}}{|\overline{w}-we^{nl}|} dg_{\HH}(w) \\
&= \int_1^{e^l}\int_{\RR} \phi\left( \frac{(x^2+y^2)(1-e^{nl})^2}{y^2e^{nl}}\right) \varepsilon^n(\mu)(-i) \frac{x+iy -e^{nl}(x-iy)}{|(x+iy)e^{nl} - (x-iy)|} \frac{dxdy}{y^2}.
\end{align*}
Next, we change variables two more times, from $\{ x, y \} $ to $\{ X:=\frac{x}{e^t}, t: = \log y \}$ and then to $\{ z:=2X\sinh(\frac{nl}{2}) \}$:
\begin{align*}
&=\int_{\RR}\int_0^l \phi \left( 4(X^2+1) \sinh^2 \left( \frac{nl}{2}\right) \right)\varepsilon^n(\mu) (-i)\frac{X(1-e^{nl})+i(1+e^{nl})}{\sqrt{X^2(e^{nl}-1)^2 + (e^{nl}+1)^2}}dXdt \\
&= \frac{l\varepsilon^n(\mu)}{2\sinh\left( \frac{nl}{2} \right)} \int_{\RR} a\left( z^2+4\sinh^2 \left( \frac{nl}{2} \right) \right) dz \cdot 2\cosh \left( \frac{nl}{2} \right) = \frac{l \varepsilon^n(\mu) v(nl)}{2\sinh \left( \frac{nl}{2}\right) }.
\end{align*}

For parabolic elements, we proceed similarly. Since $\langle\mu \rangle \setminus \HH = \langle\mu^{-1} \rangle \setminus \HH $, we can suppose that $A_{\mu} = A_{\mu^{-1}}$. After changing the variables (the same changes for both $\mu$ and $\mu^{-1}$), we can compute the integrals on the cylinder $C':=\{ w\in \HH : 0\leq \Re w <1 \}$, which is a fundamental domain of the translation by $1$ (or its inverse). Moreover, notice that integrand does not depend on the real part of $w$. Thus, the right-hand side becomes:
\begin{align*}
&\int_{C'} \left( \phi(2\cosh d(w-n,w)-2)\varepsilon^n(\mu)(-i)
\frac{w-\overline{w}+n}{|w-\overline{w}+n|} \right. \\
& \left. +\phi(2\cosh d(w+n,w)-2)\varepsilon^n(\mu^{-1})(-i)
\frac{w-\overline{w}-n}{|w-\overline{w}-n|}\right) dg_{\HH} \\
&=\int_0^{\infty} \phi \left( \frac{n^2}{y^2} \right) \varepsilon^n(\mu)(-i) \left( \frac{2iy+n}{|2iy+n|} + \frac{2iy-n}{|2iy-n|} \right)\frac{dy}{y^2} \\
&=\varepsilon^n(\mu)4 \int_0^{\infty} a \left( \frac{n^2}{y^2} \right) \frac{dy}{y^2} = \frac{\varepsilon^n(\mu)v(0)}{n}. \qedhere
\end{align*}
\end{proof}

At this point, the proof of Theorem \ref{selbergTraceFormula} is immediate. We start from the equality $(\ref{pretraceFormula})$, and then apply both Lemma \ref{calculPhi0} and Lemma \ref{calculSumaAlpha}.

\subsection{The trace formula for a larger class of functions}
We can prove the trace formula for more general functions $v$ by approximations with compactly supported functions.

\begin{definition}\label{definitieAdmissible}
We say that $v$ is an \emph{admissible} function if:
\begin{itemize}
\item it has exponential decay: $|v(x)|\leq ce^{-|x|(1/2+\epsilon)}$, for some $c>0$ and some $\epsilon > 0$;
\item There exists $\epsilon'>0$ such that $\xi^{2+\epsilon'}u(\xi) \in L^1(\RR)$.
\end{itemize}
\end{definition} 

\begin{remark*}
The second condition is equivalent with the fact that $u$ belongs to the Sobolev space $W^{2+\epsilon', 1}(\RR)$.
\end{remark*}

\begin{theorem} \label{compactCaseLargerClassOfFunctions}
Theorem \ref{selbergTraceFormula} holds true for admissible functions $v$.
\end{theorem}
\begin{proof}
Take $(v_m)_{m\in \NN}$ an increasing sequence of compactly supported functions converging pointwise to $v$. It follows that $\hat{v}_m$ also converges pointwise to $\hat{v} = u$. From the trace formula $\ref{selbergTraceFormula}$, we get a sequence of equalities which we would like to pass to the limit. To do so, we will show that each side is bounded. Let us start with the left-hand side. The derivatives of $v$ are integrable, hence $u\in \mathcal{O}(|\xi|^{-2-\epsilon'})$. If $M$ is compact, we can use Weyl's law to deduce that that $\xi_j \sim j^{1/2}$, for large $j$. Otherwise, since the spin structures along the cusp is not trivial, we can use the generalisation of Weyl's law  (see \cite[Theorem 3]{Moroianu}) to deduce the same approximation. In conclusion
\[
\sum_{j=0}^{\infty}u(\xi_j) \sim \sum_{j=1}^{\infty}j^{-1-\epsilon'/2} < \infty.
\]

To bound the right-hand side we first rewrite the two sums as a sum after all  closed geodesics on $M$ (not necessarily primitive). Let $l_k$ denote the length of the $k-$th smallest geodesic, counted with multiplicity. From Lemma \ref{marginireGeodeziceCompact}, we get that $l_k \sim \log k$, for large $k$. Thus, using the exponential decay of v, we get:
\begin{align*}
\sum_{l \in \mathcal{L}_\G} \sum_{n=1}^{\infty} \frac{l|v(nl)|}{2\sinh \left(\frac{nl}{2}\right)} & \leq 
\sum_{k=1}^{\infty} \frac{l_k |v(l_k)|}{2\sinh\left( \frac{l_k}{2} \right)} \sim \sum_{k=1}^{\infty} \log (k) k^{-1-\epsilon} <\infty. \qedhere
\end{align*}
\end{proof}


\section{The non-compact case as a limit of compact cases}

Now we have Theorem \ref{selbergTraceFormula} for both compact and non-compact surfaces of finite area. We use the same notations as before. One may wonder if the formula for the compact surfaces converges to the one for surfaces with cusps when $M$ goes through a pinching process (Definition \ref{pinchingProcess}). In this section we prove that this is indeed so. For simplicity, we present the case when the length of only one geodesic converges to $0$. 

\begin{theorem}\label{limitSelbergFormula}
Let $M$ be a compact hyperbolic surface going through a pinching process along a geodesic $\eta$ (corresponding to the conjugacy class $[\mu']$, for $\mu' \in \G$. If $\varepsilon(\mu')=-1$ (i.e. the spin structure along $\eta$ is non-trivial), then:
\begin{align*}
\sum_{[\mu] } \sum_{n=1}^{\infty} \frac{l_t(\mu)\varepsilon^n(\mu)v(nl_t(\mu))}{2\sinh \left(\frac{nl_t(\mu)}{2}\right)} \rightarrow 
\sum_{[\mu] } \sum_{n=1}^{\infty} \frac{l_0(\mu)\varepsilon^n(\mu)v(nl_0(\mu))}{2\sinh \left(\frac{nl_0(\g)}{2}\right)}-2\log(2)v(0),
\end{align*}
as $t\rightarrow 0$, where $[\mu]$ runs along all conjugacy classes of primitive elements.
\end{theorem}
\begin{proof}
There are three types of closed geodesics on $M$: those who do not intersect $\eta$, those who do intersect it transversally, and $\eta$ itself. The length of the geodesics in the first category can vary during the process, but remains bounded and does not vanish. The length of those in the second category increases to $\infty$. The sum after the conjugacy classes can be split as follows:
\[
\sum_{[\mu] } \sum_{n=1}^{\infty} \frac{l_t(\mu)\varepsilon^n(\mu)v(nl_t(\mu))}{2\sinh \left(\frac{nl_t(\mu)}{2}\right)} = 2
\sum_{n=1}^{\infty} \frac{l_t(\mu')\varepsilon^n(\mu')v(nl_t(\eta))}{2\sinh \left(\frac{nl_t(\eta)}{2}\right)}  + 
\sum_{[\mu] \neq [\mu'] } \sum_{n=1}^{\infty} \frac{l_t(\mu)\varepsilon^n(\mu)v(nl_t(\mu))}{2\sinh \left(\frac{nl_t(\mu)}{2}\right)},
\]
We study each term separately. Let us consider the function:
\begin{align*}
f:[0,\infty) \longrightarrow \RR && f(x):=\frac{xv(2x)}{\sinh(x)}.
\end{align*}
In terms of $f$, the first term becomes:
\begin{align*}
2\sum_{n=1}^{\infty} \frac{l_t(\eta) \varepsilon^n(\mu')v(nl_t(\eta))}{2\sinh\left( \frac{nl_t(\eta)}{2} \right)} &= 2\sum_{n=1}^{\infty} \frac{(-1)^n}{n}f\left( \frac{nl_t(\eta)}{2}\right) \\
&= 2\sum_{j=1}^{\infty} \frac{f(jl_t(\eta))-f \left( \frac{2j-1}{2}l_t(\eta) \right) }{2j-1} -\frac{f(jl_t(\eta))}{2j(2j-1)} \\
&=2\sum_{j=1}^{\infty} \frac{jl_t(\eta) f'(c_{j,t})}{2j(2j-1)} -\frac{f(jl_t(\eta))}{2j(2j-1)},
\end{align*}
where $c_{j,t} \in [(2j-1)l_t(\eta)/2, jl_t(\eta)]$ is given by the mean value theorem. Since both $|f(x)|$ and $|xf'(x)|$ are bounded, the above sum is bounded uniformly in $t$, and hence:
\begin{align*}
2\sum_{n=1}^{\infty} \frac{l_t(\eta) \varepsilon^n(\mu')v(nl_t(\eta))}{2\sinh\left( \frac{nl_t(\eta)}{2} \right)} \rightarrow -2\log(2)v(0) && 
\text{as } t\rightarrow 0.
\end{align*}
To tackle the second term, let us denote $l'_{m}(t)$ the length of the $m-$th smallest closed, oriented geodesic on $M$, with respect to the metric $g_t$, excluding $\eta$ and its multiples. Using the estimates from Proposition \ref{marginireGeodeziceParametru}, we get that $l'_m(t) \geq \log(m) - C$, where $C$ is bounded in $t$. Taking the sum after every closed, oriented geodesic (not only primitive ones), we obtain:
\begin{align*}
\sum_{[\mu] \neq [\mu'] } \sum_{n=1}^{\infty} \frac{l_t(\mu)\varepsilon^n(\g)v(nl_t(\mu))}{2\sinh \left(\frac{nl_t(\mu)}{2}\right)} 
& \leq 
\sum_{m=1}^{\infty} \frac{l'_m(t) |v(l'_m(t))|}{2\sinh\left( \frac{l'_m(t)}{2} \right)} \\
&\leq C \sum_{m=1}^{\infty} \log (m) m^{-1-\epsilon}.
\end{align*}

Therefore we can pass to the $t$-limit inside the initial two sums. The geodesic $\eta$ will produce the logarithmic term, the geodesics that intersect $\eta$ will vanish ($v$ has exponential decay), and all other geodesics will eventually stabilise.
\end{proof}

Thus, the right-hand side of the trace formula converges, and so we obtain the convergence of the left-hand side as well, yielding the following:
\begin{theorem}\label{convergentaSpectru}
Let $u:\RR\longrightarrow \RR$ be an admissible function (in the sense of Definition \ref{definitieAdmissible}). Consider $M$ a hyperbolic surface, with a spin structure, which undergoes a pinching process and suppose that the spin structure along the pinched geodesics is not trivial. Then, as $t$ converges to $0$, we have that:
\begin{align*}
\tr\left( u(\D_t^2) \right) \rightarrow \tr\left( u(\D_0^2) \right),
\end{align*}
where $\D_t$ denotes the Dirac operator on $M$ for the metric $g_t$.
\end{theorem}


\section{Applications}
Four applications can be obtained from Theorems \ref{selbergTraceFormula} and \ref{limitSelbergFormula}. Most of them will make use of the following family of functions:
\begin{align}\label{heatFamilyForSelberg}
v_T:\RR\longrightarrow \RR, && v_T(x):=\frac{e^{-x^2/4T}}{\sqrt{4\pi T}},
\end{align}
where $T>0$. Clearly $v_T$ is admissible. Moreover, its Fourier transform is given by:
\begin{align*}
u_T :\RR \longrightarrow \RR, && u_T(\xi) := e^{-T\xi^2}.
\end{align*}

\subsection{Huber's Theorem for the Dirac operator}
Consider $M$ and $N$ two hyperbolic surfaces of finite area endowed with non-trivial spin structures.
\begin{proof}[Proof of Theorem \ref{isospectralityTheorem}]
The proof follows the lines presented in \cite{Buser}. Suppose they have the same spectrum of the Dirac operator. We want to show that they also have the same area, number of cusps, and length spectrum. We apply Theorem \ref{selbergTraceFormula} for both $M$ and $N$, using the family $(\ref{heatFamilyForSelberg})$ and take the difference:
\begin{align*}
0&=\left( \area(M) - \area(N) \right) \frac{1}{4\pi}\int_{\RR} e^{-T\xi^2} \xi\coth(\pi \xi)d\xi \\
&+ 
\sum_{[\mu]}\sum_{n=1}^{\infty} \frac{l(\mu)\varepsilon_M^n(\mu)e^{-n^2l^2(\mu)/4T}}{2\sinh \left( \frac{nl(\mu)}{2}\right)\sqrt{4\pi T}}
-
\sum_{[\nu]}\sum_{n=1}^{\infty} \frac{l(\nu)\varepsilon_N^n(\nu)e^{-n^2l^2(\nu)/4T}}{2\sinh \left( \frac{nl(\nu)}{2}\right)\sqrt{4\pi T}} \\
&-
\frac{(k_M-k_N)\log(2)}{\sqrt{4\pi T}}.
\end{align*}
First, let us denote:
\begin{align}\label{heatTraceIntegralFunction}
I(T):=\int_{\RR} e^{-T\xi^2} \xi\coth(\pi \xi)d\xi.
\end{align}
Changing the variables to $\xi'=\xi\sqrt{T}$ one can easily see that 
$\lim_{T\rightarrow 0} TI(T)= 1$. Second, the length of primitive elements is bounded from below by the systole (i.e. the shortest geodesic). Moreover, the map 
\begin{align}\label{exponentialDecayGeodesics}
T \mapsto \frac{l(\mu)\varepsilon^n(\mu)e^{-n^2l^2(\mu)/4T}}{2\sinh \left( \frac{nl(\mu)}{2}\right)\sqrt{4\pi T}},
\end{align}
is non-decreasing for $T$ close to $0$, hence, the sums in the right-hand side decrease to $0$ exponentially fast as $T\rightarrow 0$. Therefore, $M$ and $N$ have the same area and the same number of cusps (and hence, the same Betti number $b_1$). Going forward, suppose $[\mu_1]$ and $[\nu_1]$ are the shortest geodesic on $M$ and $N$ respectively. If $l(\mu_1)<l(\nu_1)$, then we chose $v$ a function with compact support which takes value $1$ on $l(\mu_1)$ and vanishes on the length of all other geodesics on either $M$ or $N$. Applying again Theorem \ref{selbergTraceFormula} for this function we get a contradiction. It follows that $l(\mu_1)=l(\nu_1)$. Moreover they must have the same multiplicity and the same sign associated by the two class functions $\varepsilon_M$ and $\varepsilon_N$. We continue the same way to finish the direct implication. For the converse, suppose the two surfaces have the same primitive length spectrum and the class functions coincide on it. We use again the family of functions $(\ref{heatFamilyForSelberg})$, however, this time we will look at the limit $T\rightarrow \infty$. The trace formula yields:
\begin{align*}
0
&=
\sum_{j=0}^{\infty} e^{-T\lambda^M_j} 
-
\sum_{j=0}^{\infty} e^{-T\lambda^N_j}\\
&- 
\frac{\area (M)-\area(N)}{4\pi}\int_{\RR} e^{-T\xi^2} \xi\coth(\pi \xi)d\xi\\
&+ 
\frac{(k_M-k_N)\log(2)}{\sqrt{4\pi T}}.
\end{align*}
When $T\rightarrow \infty$, the power series of the second term is:
\begin{align*}
I(T)=\frac{1}{\sqrt{\pi T}} + \frac{1}{6}\sqrt{ \frac{\pi^3}{T^3}} + \mathcal{O}\left( \frac{1}{\sqrt{T^5}}\right).
\end{align*}
Hence, $0$ has the same multiplicity in the spectra of the two surfaces, and the two surfaces have the same number of cusps and the same area. After removing the null eigenvalues, their respective heat traces are equal. Dividing by the next smallest eigenvalues, we deduce it must appear in both the Dirac spectrum on $M$ and on $N$, with the same multiplicity. This concludes the proof.
\end{proof}
As a conclusion, one can read both the number of cusps and the area (thus, the first Betti number) of a surface in either its Dirac spectrum or in its primitive length spectrum.

\subsection{Heat trace asymptotics}
Consider $M$ a hyperbolic surface with $k$ cusps. Suppose we pick a spin structure which satisfies the hypothesis of Theorem \ref{selbergTraceFormula}.
\begin{proof}[Proof of Theorem \ref{heatTraceAsymptotic}]
For the family of functions $(\ref{heatFamilyForSelberg})$, the trace formula reads:
\begin{align*}
\sum_{j=0}^{\infty} e^{-T\lambda_j} 
&= 
\frac{\area (M)}{4\pi}\int_{\RR} e^{-T\xi^2} \xi\coth(\pi \xi)d\xi + \sum_{[\mu]}\sum_{n=1}^{\infty} \frac{l(\mu)\varepsilon^n(\mu)e^{-n^2l^2(\mu)/4T}}{2\sinh \left( \frac{nl(\mu)}{2}\right)\sqrt{4\pi T}}-\frac{k\log(2)}{\sqrt{4\pi T}},
\end{align*}
where $[\mu]$ runs over all classes of primitive, hyperbolic elements. As before, the map in $(\ref{exponentialDecayGeodesics})$ is non-increasing. Hence, the second term in the right-hand side decreases exponentially fast. For the integral $I(T)$ defined above $(\ref{heatTraceIntegralFunction})$ consider the difference:
\begin{align*}
I(T) - \frac{1}{T}=\int_{\RR} e^{-T\xi^2}\left( \xi\coth(\pi \xi)-|\xi|\right) d\xi.
\end{align*}
Note that the map $\xi \mapsto \xi\coth(\pi \xi)-|\xi|$ vanishes exponentially fast at infinity. Thus:
\begin{align*}
I(T) - \frac{1}{T}\sim \sum_{m=0}^{\infty}a_m T^m; && a_m=\int_{\RR}e^{-T\xi^2 }(-1)^m\xi^{2m}\left( \xi\coth(\pi \xi)-|\xi|\right) d\xi,
\end{align*}
meaning that 
\begin{align*}
\left| I(T)-\frac{1}{T}-\sum_{m=1}^n a_mT^m \right| \in \mathcal{O}(T^{n+1}).
\end{align*}

Therefore, as $T\searrow 0$:
\[
\sum_{j=0}^{\infty} e^{-T\lambda_j} \sim \frac{\area (M)}{4 \pi T} -\frac{k \log(2)}{\sqrt{4\pi T}} + \sum_{m=0}^{\infty}\frac{a_m\area (M) }{4\pi}T^m.
\qedhere
\]
\end{proof}

Note that the only term in the asymptotic that depends on the surface is $\frac{k \log(2)}{\sqrt{4\pi T}}$. The other terms come from the action of $\D$ on the hyperbolic plane and are the same for any surface.

\subsection{Uniform Weyl law}
Suppose that $M$ goes through a pinching process as in Definition \ref{pinchingProcess}, producing at $t=0$ a hyperbolic surface with $2k$ cusps. Pinching along more than one geodesic adds no conceptual difficulty to the proof, thus we will present the case $k=1$. Denote $\eta$ the pinched geodesic and $[\mu'] $ its corresponding conjugacy class.
\begin{proof}[Proof of Theorem \ref{uniformWeylLaw}]
From the hypothesis, we know that $\varepsilon(\mu')=-1$. By $\lambda_j(t)$ we denote the $j$-th largest eigenvalue (counted with multiplicity) of $\D_t^2$, the Dirac operator corresponding to the metric $g_t$. Remember that $\D_t^2$ has two times the spectrum of $\D^-\D^+$. Again, applying the trace formula for the family of functions $(\ref{heatFamilyForSelberg})$, we get:
\begin{align*}
\sum_{j=0}^{\infty} e^{-T\lambda_j(t)} 
&= \frac{\area (M)}{2\pi}\int_{\RR}\xi e^{-T\xi^2} \coth(\pi \xi)d\xi 
+ 4\sum_{n=1}^{\infty}\frac{l(\mu')(-1)^n e^{-n^2l^2(\mu')/4T}}{2\sinh \left( \frac{nl(\mu')}{2}\right)\sqrt{4\pi T}}\\
&+ 
2\sum_{[\mu]\neq [\mu']}\sum_{n=1}^{\infty} \frac{l(\mu)\varepsilon^n(\mu)e^{-n^2l^2(\mu)/4T}}{2\sinh \left( \frac{nl(\mu)}{2}\right)\sqrt{4\pi T}}.
\end{align*}
We have seen in the proof of Theorem \ref{limitSelbergFormula} that, as $t\rightarrow 0:$
\begin{align*}
4\sum_{n=1}^{\infty}\frac{l(\mu')(-1)^n e^{-n^2l^2(\mu')/4T}}{2\sinh \left( \frac{nl(\mu')}{2}\right)\sqrt{4\pi T}} \rightarrow -\frac{2\log(2)}{\sqrt{\pi T}}.
\end{align*}
Combining this convergence with Theorem \ref{heatTraceAsymptotic} we can deduce the following convergence, uniformly in $t$:
\begin{align}\label{heatTraceLimit}
\lim_{T\rightarrow 0} T \sum_{j=0}^{\infty} e^{-T\lambda_j(t)} =\frac{\area (M)}{2\pi },
\end{align}
From here, we finish using  Karamata's theorem. However, for completeness, we will include a sketch of the proof as well. Denote $\alpha:=\frac{\area (M)}{2\pi}$. From $(\ref{heatTraceLimit})$ one can easily see that:
\begin{align*}
\lim_{T\rightarrow 0} T\sum_{j=0}^{\infty}e^{-T\lambda_j(t)} P(e^{-T\lambda_j(t)}) = \alpha\int_0^1 P(s)ds,
\end{align*}
where $P(s):=\sum_{m=0}^r c_m s^m$ is an arbitrary polynomial. Clearly this convergence is uniform in $t$. Since polynomials are dense in continuous functions on $[0,1]$, we obtain:
\begin{align}\label{heatTraceFunctionLimit}
\lim_{T\rightarrow 0} T\sum_{j=0}^{\infty}e^{-T\lambda_j(t)} f(e^{-T\lambda_j(t)}) = \alpha\int_0^1 f(s)ds,
\end{align}
uniformly in $t$, where $f:[0,1]\longrightarrow \RR$ is the continuous function defined by:
\begin{align*}
f(s) = 
\begin{cases}
0 & s\in [0,a];\\
\frac{s-a}{b-a} & s\in [a,b];\\
\frac{1}{s} & s\in [b,1],
\end{cases}
\end{align*}
where $a<b$ are two numbers between $0$ and $1$. Hence, for $T$ small enough, $(\ref{heatTraceFunctionLimit})$ implies:
\begin{align*}
T N_t\left(- \frac{\log a}{T} \right) \geq T\sum_{j=0}^{\infty}e^{-T\lambda_j(t)} f(e^{-T\lambda_j(t)}) \geq T N_t\left(- \frac{\log b}{T} \right).
\end{align*}
Recall that $N_t(\lambda)$ is the number of $\D_t^2$-eigenvalues smaller or equal than $\lambda$. Subtracting $\alpha \int_0^1 f(s) ds$ in the middle term of the inequality above, we get:
\begin{align*}
T N_t\left(- \frac{\log a}{T} \right) + \alpha \log b &\geq T\sum_{j=0}^{\infty}e^{-T\lambda_j(t)} f(e^{-T\lambda_j(t)}) - \alpha \int_0^1 f(s) ds \\
&\geq T N_t\left(- \frac{\log b}{T} \right) +\alpha \log a,
\end{align*}
which, dividing by $-\log b$ can be rewritten:
\begin{align*}
\frac{T}{-\log a} \frac{\log a}{\log b} N_t\left(- \frac{\log a}{T} \right) - \alpha &\geq \frac{1}{-\log b}\left( T\sum_{j=0}^{\infty}e^{-T\lambda_j(t)} f(e^{-T\lambda_j(t)}) - \alpha \int_0^1 f(s) ds \right) \\
&\geq \frac{T}{-\log b} N_t\left(- \frac{\log b}{T} \right) -\alpha \frac{\log a}{\log b}.
\end{align*}
Since $(\ref{heatTraceFunctionLimit})$ was uniformly in $t\in [0,1]$, when $T\searrow 0$, we obtain two inequalities:
\begin{align*}
&\frac{\log a}{\log b} \liminf_{\lambda\rightarrow 0} \frac{N_t(\lambda)}{\lambda} \geq \alpha,\\
&\alpha \frac{\log a}{\log b} \geq \limsup_{\lambda\rightarrow 0} \frac{N_t(\lambda)}{\lambda},
\end{align*}
both of them uniformly in $t\in [0,1]$. Finally, making $b\rightarrow a$ finishes the proof. 
\end{proof}

\subsection{Selberg Zeta function for non-compact surfaces}
In this section we extend the definition of the Selberg zeta function to non-compact, hyperbolic surfaces of finite area, endowed with a non-trivial spin structure. We are able to show that this function extends analytically to the whole complex plane and, moreover, it respects a functional equation given by a symmetry with respect to the point $1/2\in \CC$. The proof is rather standard, so we are going to sketch it.

Let $M = \G \setminus \HH$ be complete hyperbolic surface of finite area, having $k$ cusps. Further suppose $M$ is endowed with a non-trivial spin structure and let $\varepsilon : \G \longrightarrow \{ \pm 1
\}$ be the associated class function (definition $\ref{caracter}$). Consider the admissible function:
\begin{align*}\label{zetaAdmissibleFunction}
v:\RR \longrightarrow \RR, && v(t):= \frac{e^{-|t|\left(s-1/2\right)}}{2s-1}-\frac{e^{-|t|\left(s_0-1/2 \right)}}{2s_0-1},
\end{align*}
where $s, s_0 \in \CC$ are two complex numbers with $\Re(s), \Re(s_0) >1$ ($s_0$ will be fixed). The Fourier transform of $v$ is the map:
\begin{align*}
u(\xi)=\frac{1}{\xi^2+\left( s-\tfrac{1}{2}\right) ^2}-\frac{1}{\xi^2+\left(s_0-\tfrac{1}{2}\right) ^2},
\end{align*}
a difference between two resolvents. Note that we had to take this difference in order for $v$ to be admissible. Let $Z=Z_{\varepsilon}$ be the Selberg zeta function defined in formula $(\ref{formulaZeta})$
with the respect to the class function $\varepsilon$. Since $\varepsilon$ is fixed, in what follows we will drop it from the notation of the zeta function. A direct computation leads to the following identity:

\begin{lemma}\label{formulaDerivataLogaritmica}
\[
\sum_{[\mu]}\sum_{n=1}^{\infty} \frac{ l(\mu) \varepsilon^n(\mu) e^{-nl(\mu)\left( s-1/2 \right) }}{2 \sinh \left( \tfrac{nl(\mu)}{2} \right) } = \frac{Z'}{Z}(s).
\]
\end{lemma}

Applying Theorem $\ref{selbergTraceFormula}$ for $v$ and using the above lemma we get that:
\begin{align}\label{derivataLogaritmicaZ}
\begin{split}
\frac{Z'}{Z}(s)
&= 
\frac{2s-1}{2s_0-1}\frac{Z'}{Z}(s_0)+(2s-1)\sum_{j=0}^{\infty} \left( \frac{1}{\xi_j^2+\left( s-\tfrac{1}{2}\right) ^2}-\frac{1}{\xi_j^2+\left(s_0-\tfrac{1}{2}\right) ^2} \right)  \\
&-\frac{\area(M)}{4\pi}\int_{\RR} \xi \coth (\pi \xi)\left(  \frac{2s-1}{\xi^2+\left( s-\tfrac{1}{2}\right) ^2}-\frac{2s-1}{\xi^2+\left(s_0-\tfrac{1}{2}\right) ^2} \right) d\xi \\
&+ k\log(2)\left(1- \frac{2s-1}{2s_0-1} \right).
\end{split}
\end{align}
To show that $\frac{Z'}{Z}$ extends meromorphically on $\CC$ with simple poles and positive integer residues, we need two simple lemmas. 

\begin{lemma}
The term
\begin{align*}
(2s-1)\sum_{j=0}^{\infty} 
\left( \frac{1}{\xi_j^2+\left( s-\tfrac{1}{2}\right) ^2}-\frac{1}{\xi_j^2+\left(s_0-\tfrac{1}{2}\right) ^2} \right),
\end{align*}
is meromorphic in $s\in \CC$, with simple poles at $\tfrac{1}{2} \pm is_j$. The residue is $m (\lambda_j)$, the multiplicity of the eigenvalue $\lambda_j$. If $0$ lies inside the spectrum of $\D^- \D^+$, the residue at $s=\tfrac{1}{2}$ will be $2m(0)$.
\end{lemma}
\begin{proof}
Clearly $\frac{2s-1}{\xi_j^2+\left( s-\tfrac{1}{2}\right) ^2}$ has simple poles at $s=\tfrac{1}{2} \pm i\xi_j$ with residue $1$. Since the sum is absolutely convergent the conclusion follows.
\end{proof}

\begin{lemma}\label{rescriereDerivataLogaritmica}
The third term in the right-hand side of $(\ref{derivataLogaritmicaZ})$ can be rewritten as:
\begin{align*}
&\frac{\area(M)}{4\pi}\int_{\RR} \xi \coth (\pi \xi)\left(  \frac{2s-1}{\xi^2+\left( s-\tfrac{1}{2}\right) ^2}-\frac{2s-1}{\xi^2+\left(s_0-\tfrac{1}{2}\right) ^2} \right) d\xi  \\
&=
\frac{\area(M)(2s-1)}{4\pi}\left( -(\pi \tan(\pi s) -\pi \tan (\pi s_0)) - \sum_{n=1}^{\infty}  \frac{2n}{\left( s-\tfrac{1}{2} \right)^2-n^2}-\frac{2n}{\left( s_0-\tfrac{1}{2} \right)^2-n^2} \right) \\
&=
-\frac{\area(M)}{2\pi} \left( \frac{2s-1}{2s_0-1} -1 + \sum_{n=1}^{\infty}\left( \frac{2s-1}{s_0-\frac{1}{2}+n} -\frac{2s-1}{s-\frac{1}{2}+n} \right)\right).
\end{align*}\label{integrala}
\end{lemma}
\begin{proof}
The first equality can be seen as follows. The integrand has simple poles at $\xi=in$, $\xi=i\left( s-\tfrac{1}{2}\right)$ and at $\xi=i\left( s_0-\tfrac{1}{2}\right)$, for $n$  a positive integer. Moving up the integral line we can see that its value converges to $0$, since the integrand's absolute value is bounded by $cn^{-3}$, for some constant $c>0$. Hence, the integral can be computed as the sum of the residues. To obtain the final equality, we apply $(\ref{identitateTanh})$ for $\tan(\pi s)= -i\tanh(i\pi s)$.
\end{proof}

Thus, $(\ref{derivataLogaritmicaZ})$ implies that the logarithmic derivative of $Z$ extends meromorphically to $\CC$. Moreover, it has simple poles with positive integers residues, hence $Z$ extends analytically to $\CC$. We have just proved the following:

\begin{theorem} \label{extindereDerivataLogaritmica}
The logarithmic derivative of the zeta function corresponding to $\D^-\D^+$ on a complete hyperbolic surface of finite area extends meromorphically to $\CC$ with simple poles and positive integer residues:
\begin{itemize}
\item $z=\tfrac{1}{2} \pm i\xi$ for all $\xi^2=\lambda_j\in \spec(\D^{-}\D^{+})$ with residue $m(\lambda_j)$;
\item $z=\tfrac{1}{2}$ with residue $2m( 0) $, if $0\in \spec(\D^{-}\D^{+})$;
\item $z=-n+\tfrac{1}{2}$ for all positive integers $n$, with residue $\tfrac{n\area(M)}{\pi}$, (which is an integer, from Gauss-Bonnet).
\end{itemize}
It follows that $Z$ extends holomorphically to $\CC$.
\end{theorem}

\begin{theorem}
In the context of Theorem \ref{extindereDerivataLogaritmica}, the function:
\begin{align*}
\tilde{Z}(s)
:=
Z(s)\exp\left( -\frac{\area (M)}{2} \int_{1/2}^{s} \left(w-\frac{1}{2}\right)\tan(\pi w)dw \right)2^{(1-2s)k},
\end{align*}
satisfies the following functional equation:
\begin{align*}
\tilde{Z}(s) = \tilde{Z}(1-s),
\end{align*}
for any $s\in \CC$.
\end{theorem}
\begin{proof}
We start from formula $(\ref{derivataLogaritmicaZ})$ and make use of Lemma $\ref{rescriereDerivataLogaritmica}$:
\begin{align*}
\frac{Z'}{Z}(s) + \frac{Z'}{Z}(1-s)
&=
-\frac{\area(M)}{4\pi}(2s-1)\left( -\pi \tan(\pi s) + \pi \tan(\pi-\pi s) \right) + 4k\log(2)\\
&=\area(M)\left( s-\frac{1}{2} \right)\tan(\pi s)+4k\log(2).
\end{align*}
Integrating, we obtain:
\begin{align*}
\log \left( \frac{Z(s)}{Z(1-s)} \right)
=
\int_{1/2}^{s}\area (M)\left( w-\frac{1}{2} \right)\tan(\pi w)dw + (2s-1)2k\log(2),
\end{align*}
and the conclusion follows immediately.
\end{proof}
\noindent
In the compact case, the existence of such a functional equation was alluded in \cite{BolteStiepanSelbergForDirac} but not written down explicitly.

From the definition, it is clear that $Z$ depends on the surface $M$ and, consequently, on the hyperbolic metric $g$. To avoid confusion, we shall denote this dependency by rewriting the function as $Z(s,(M,g))$. A natural question that arises is what happens with the zeta function when a compact surface goes through a pinching process. In other words, what happens with the zeta function along a path in the Teichm\"uller space going to the boundary. The lemma we present below follows the ideas of a similar result in \cite{SchulzeZetaLaplacian}.

Suppose the compact surface $M$ goes through a pinching process along one geodesic $\eta$. Further suppose that $\varepsilon(\eta)=-1$ for all $t$ (recall that this hypothesis means that the spectrum of the Dirac operator is discrete at the boundary). Then, the zeta function diverges. The divergent term, when $t\rightarrow 0$, is given by the product:
\begin{align*}
Z_{\eta}(s,(M,g_t)):=\prod_{k=0}^{\infty}\left( 1+ e^{-(s+k)l_t(\eta)} \right)^2,
\end{align*}
for $\Re(s)>1$. We take the square because the same product appears twice in the definition of $Z$, one for each orientation of $\eta$. 
The asymptotic behaviour of $Z$ at the boundary of the Teichm\" uller space is given by the following result:
\begin{lemma}\label{ZetaExponentialGrowth}
For $s\in \CC$ with $\Re(s)>1$,
\begin{align*}
Z_\eta(s,(M,g_t))e^{-\pi^2/6l_t(\eta)} \rightarrow 2^{1-2s},
\end{align*}
as $t$ converges to $0$.
\end{lemma}
\begin{proof}
Rewrite the logarithmic derivative using the Taylor series of the exponential:
\begin{align*}
\log Z_{\eta}(s,(M,g_t)) &= \sum_{k=0}^{\infty} 2\log \left( 1+ e^{-(s+k)l_t(\eta)} \right) = \sum_{k=0}^{\infty}\sum_{n=1}^{\infty} -\frac{2(-1)^n}{n}e^{-(s+k)nl_t(\eta)} \\
&=\sum_{n=1}^{\infty} -\frac{2(-e^{-sl_t(\eta)})^n}{n}\frac{1}{1-e^{-nl_t(\eta)}}.
\end{align*}
Now, we split the sum into four parts, using
\begin{align*}
\frac{1}{1-e^{-nl_t(\eta)}} =\frac{1}{2}+ \frac{1}{nl_t(\eta)} +\left(\frac{1}{e^{nl_t(\eta)}-1} -\frac{1}{nl_t(\eta)} +\frac{1}{2} \right).
\end{align*}
Hence, we obtain:
\begin{align}
\log Z_{\eta}(s,(M,g_t)) &= \sum_{n=1}^{\infty} -\frac{(-e^{-sl_t(\eta)})^n}{n} \label{T1} \\
&- \frac{2}{l} \sum_{n=1}^{\infty} \frac{(-e^{-sl_t(\eta)})^n}{n^2} \label{T2} \\
&+\sum_{k=1}^{\infty}2l_t(\eta) \frac{e^{sl_t(\eta)(2k-1)}}{(2k-1)l_t(\eta)}\left(\frac{1}{e^{(2k-1)l_t(\eta)}-1} -\frac{1}{(2k-1)l_t(\eta)} +\frac{1}{2} \right) \label{T3} \\
&-\sum_{k=1}^{\infty}2l_t(\eta) \frac{e^{sl_t(\eta)2k}}{2kl_t(\eta)}\left(\frac{1}{e^{2kl_t(\eta)}-1} -\frac{1}{2kl_t(\eta)} +\frac{1}{2} \right) \label{T4}
\end{align}
The first term $(\ref{T1})$ equals the logarithm: $\log(1+e^{-sl_t(\eta)})$, which is smooth at $t=0$. Both $(\ref{T3})$ and $(\ref{T4})$ are Riemann sums. As $t$ approaches $0$, the term $(\ref{T3})$ converges to
\[
\int_0^{\infty} \frac{e^{-sx}}{x}\left(\frac{1}{e^x-1}-\frac{1}{x}+\frac{1}{2}\right)dx,
\]
and $(\ref{T4})$ converges to the opposite of the above integral. 
Since $t\mapsto \tfrac{1}{e^x-1}-\tfrac{1}{x}+\tfrac{1}{2}$ is of order $\mathcal{O}(x^2)$ near $0$, this integral is finite. Hence, at $t=0$, $(\ref{T3})$ and $(\ref{T4})$ will cancel out.
Finally, the second term $(\ref{T2})$ yields the dilogarithm, $-\frac{2}{l_t(\eta)} \li_2(-e^{-sl_t(\eta)})$, with a simple pole at $l_t(\eta)=0$. Using the (obvious) equalities:
\begin{align*}
\li_2(-1)=-\frac{\pi^2}{12}, &&
\li_2'(x) = -\frac{\log(1-x)}{x},
\end{align*}
we can deduce that:
\[
-\frac{2}{l_t(\eta)}\li_2(-e^{-sl_t(\eta)}) - \frac{\pi^2}{6l_t(\eta)} \rightarrow -2s\log(2).
\]
Finally, putting everything together, we obtain as $t\rightarrow 0$:
\[
\log Z_{\eta} (s,(M,g_t)) -\frac{\pi^2}{6l_t(\eta)} \rightarrow (1-2s)\log(2).
\qedhere
\]
\end{proof}
\subsection{The convergence of the Selberg Zeta function under a pinching process}
The previous result implies that $Z$ diverges when approaching the boundary of the Teich\" muller space. However, this is not the case for its logarithmic derivative. Remember that $M$ goes through a pinching process along the simple geodesic $\eta$. One can easily see that $\frac{Z'}{Z} (s, (M,g_t) \rightarrow \frac{Z'}{Z} (s, (M,g_0)) - 2\log (2)$ as $t$ goes to $0$, uniformly on compacts in the half-plane $\Re(s)>1$. Indeed, if $[\mu']$ is the primitive class associated to the pinched geodesic $\eta$, rewrite Lemma \ref{formulaDerivataLogaritmica}
\begin{align*}
\frac{Z'}{Z}(s,(M,g_t))
&=
2\sum_{n=1}^{\infty}\frac{l_t(\mu')\varepsilon^n(\mu')e^{-nl_t(\mu')\left( s-1/2 \right)}}{2 \sinh \left( \tfrac{nl_t(\mu')}{2} \right) }
+
\sum_{[\g]\neq [\mu']}\sum_{n=1}^{\infty} \frac{ l_t(\g) \varepsilon^n(\g) e^{-nl_t(\g)\left( s-1/2 \right) }}{2 \sinh \left( \tfrac{nl_t(\g)}{2} \right) },
\end{align*}
where $[\g]$ runs along all conjugacy classes of primitive elements. As in the proof of Theorem \ref{limitSelbergFormula}, when $t\rightarrow 0$, the right hand side has a limit:
\begin{align*}
&2\sum_{n=1}^{\infty}\frac{l_t(\mu')\varepsilon^n(\mu')e^{-nl_t(\mu')\left( s-1/2 \right)}}{2 \sinh \left( \tfrac{nl_t(\mu')}{2} \right)} \rightarrow -2\log(2); \\
&\sum_{[\g]\neq [\mu']}\sum_{n=1}^{\infty} \frac{ l_t(\g) \varepsilon^n(\g) e^{-nl_t(\g)\left( s-1/2 \right) }}{2 \sinh \left( \tfrac{nl_t(\g)}{2} \right)} 
\rightarrow
\sum_{[\g] \text{ hyperbolic}}\sum_{n=1}^{\infty} \frac{ l_0(\g) \varepsilon^n(\g) e^{-nl_0(\g)\left( s-1/2 \right) }}{2 \sinh \left( \tfrac{nl_0(\g)}{2} \right) }.
\end{align*}
Writing $(\ref{derivataLogaritmicaZ})$ in the compact case (i.e. $k=0$)
\begin{align}\label{derivataLogaritmicaMeromorphicExtension}
\begin{split}
\frac{Z'}{Z}(s,(M,g_t)) 
&=
\frac{2s-1}{2s_0-1}\frac{Z'}{Z}(s_0,(M,g_t))\\
&+
(2s-1)\sum_{j=0}^{\infty} \left( \frac{1}{\xi_j^2(t)+\left( s-\tfrac{1}{2}\right) ^2}-\frac{1}{\xi_j^2(t)+\left(s_0-\tfrac{1}{2}\right) ^2} \right) \\
&-
\frac{\area(M)}{4\pi}\int_{\RR} \xi \coth (\pi \xi)\left(  \frac{2s-1}{\xi^2+\left( s-\tfrac{1}{2}\right) ^2}-\frac{2s-1}{\xi^2+\left(s_0-\tfrac{1}{2}\right) ^2} \right) d\xi,
\end{split}
\end{align}
one can see that $\frac{Z'}{Z}(s,(M,g_t))$ has a meromorphic extension on the whole complex plane. We show below that:
\begin{align*}
\lim_{t\rightarrow 0} \frac{Z'}{Z} (s, (M,g_t) = \frac{Z'}{Z} (s, (M,g_0)) - 2\log (2),
\end{align*}
holds for $s$ in compacts on $\CC$ which do not contain the poles of the limiting zeta function. The way to tackle such a convergence is by combining the convergence of the spectrum under degenerations, a result due to Pf\" affle \cite{Pfaffle}, together with the uniform Weyl law  (Theorem $\ref{uniformWeylLaw}$). Translated into our context ($M$ goes through a pinching process), Pf\" affle's result reads:
\begin{theorem}[\cite{Pfaffle}, Theorem 1]\label{resultPfaffle}
For $\epsilon>0$ and $m$ a positive integer there exists $t_{\epsilon, m}>0$ such that for any $0 \leq t \leq t_{\epsilon, m}$ we have $|\xi_j(t) - \xi_j(0)| \leq \epsilon$, where $0 \leq j \leq m$ and $\xi_j(t)$ is the $j$-th eigenvalue of $\D_t$.
\end{theorem}
The continuous variation of the spectrum is a delicate task. In literature there exist pseudodifferential calculi encompassing degenerations similar to our pinching limits at least on a formal level e.g. McDonald \cite{McDonald}, Mazzeo and Melrose \cite{MelroseMazzeoEta}, Albin, Rochon and Sher \cite{Federic}. Since those methods are quite far from the current paper we will not delve in this direction. Let us mention an ongoing investigation in this direction by Cipriana Anghel \cite{Cipi}. 

Back to our problem, let us prove the following lemma:
\begin{lemma}\label{uniformEigenvaluesSumBoundAtInfinity}
\begin{align*}
\lim_{p \rightarrow \infty}
\sum_{j=p+1}^{\infty} \left| \frac{1}{\xi_j^2(t)+\left( s-\tfrac{1}{2}\right) ^2}-\frac{1}{\xi_j^2(t)+\left(s_0-\tfrac{1}{2}\right) ^2} \right| 
= 0,
\end{align*}
uniformly in $t \in [0,1]$ and $s$ in a compact set in $\CC$.
\end{lemma}
\begin{proof}
By direct computations one can see that:
\begin{align*}
\sum_{j=0}^{\infty} \left| \frac{1}{\xi_j^2(t)+\left( s-\tfrac{1}{2}\right) ^2}-\frac{1}{\xi_j^2(t)+\left(s_0-\tfrac{1}{2}\right) ^2} \right|
=
\sum_{j=0}^{\infty} \left| \frac{ \left( s_0-\tfrac{1}{2}\right)^2 - \left( s-\tfrac{1}{2}\right)^2}{ \left( \xi_j^2(t)+\left( s-\tfrac{1}{2}\right)^2 \right) \left( \xi_j^2(t)+\left( s_0-\tfrac{1}{2}\right)^2 \right)} \right|.
\end{align*}
Theorem $\ref{uniformWeylLaw}$, implies:
\begin{align*}
\lim_{j \rightarrow \infty} \frac{\left( \xi_j^2(t)+\left( s-\tfrac{1}{2}\right)^2 \right) \left( \xi_j^2(t)+\left( s_0-\tfrac{1}{2}\right)^2 \right)}{j^2} =  \frac{\area (M) ^2}{4\pi^2},
\end{align*}
uniformly in $t\in [0,1]$ and $s$ in a compact set. Hence there exists a constant $C$ and $j_0$ such that for $j\geq j_0$:
\[
\left| \frac{1}{\xi_j^2(t)+\left( s-\tfrac{1}{2}\right) ^2}-\frac{1}{\xi_j^2(t)+\left(s_0-\tfrac{1}{2}\right) ^2} \right|
\leq 
  \frac{C}{j^2},
\]
for any $t$ and $s$ as above.
\end{proof}
We are now able to prove the convergence of the logarithmic derivative:
\begin{theorem}\label{logarithmicDerivativeConvergence}
The family of meromorphic functions $\frac{Z'}{Z}(s,(M,g_t))$, with $s\in \CC$ and $t\in (0,1]$ converges uniformly, as $t\rightarrow 0$, to $\frac{Z'}{Z}(s,(M,g_0)) - 2\log(2)$, on compacts which avoid the poles of the limit function.
\end{theorem}
\begin{proof}
We start from formula $(\ref{derivataLogaritmicaMeromorphicExtension})$. To deal with the first term in the right-hand side, let us fix $s_0\in \CC$ with $\Re(s_0)>1$. We have noted above that for such $s_0$,
\begin{align*}
\lim_{t\rightarrow 0 } \frac{2s-1}{2s_0-1} \frac{Z'}{Z}(s_0,(M,g_t)) = \frac{2s-1}{2s_0-1} \frac{Z'}{Z}(s_0,(M,g_0)) - 2\log(2)\frac{2s-1}{2s_0-1},
\end{align*}
uniformly in $t$ and in $s$ in a compact set. From Lemma \ref{rescriereDerivataLogaritmica}, the third term does not depend on the parameter $t$ and is absolutely convergent for $s$ in a compact which does not contain the poles of $\frac{Z'}{Z}(s,(M,g_0))$. Thus, we are left with the convergence of the second term in the right-hand side of $(\ref{derivataLogaritmicaMeromorphicExtension})$. Take $\epsilon >0$ and a compact $K\subset \CC$ as required in the hypothesis. By Lemma \ref{uniformEigenvaluesSumBoundAtInfinity} there exists $p$ a positive integer such that:
\begin{align*}
\sum_{j=p+1}^{\infty}\left| \frac{1}{\xi_j^2(t)+\left( s-\tfrac{1}{2}\right) ^2}-\frac{1}{\xi_j^2(t)+\left(s_0-\tfrac{1}{2}\right) ^2} \right| <\epsilon,
\end{align*}
for every $t\in [0,1]$ and $s\in K$. Moreover, for $\frac{\epsilon}{p+1}$ and $p$, Pf\" affle's result (Theorem \ref{resultPfaffle}) implies the existence of $t_{\epsilon,p}$ such that $|\xi_j^2(t) - \xi_j^2(0)| < \frac{\epsilon}{p+1}$ for every $0\leq t \leq t_{\epsilon,p}$. Combining these two inequalities we obtain:
\begin{align*}
\sup_{s\in K} & \left| \sum_{j=0}^{\infty} \left( \frac{1}{\xi_j^2(t)+\left( s-\tfrac{1}{2}\right) ^2}-\frac{1}{\xi_j^2(t)+\left(s_0-\tfrac{1}{2}\right) ^2}- \frac{1}{\xi_j^2(0)+\left( s-\tfrac{1}{2}\right) ^2}+\frac{1}{\xi_j^2(0)+\left(s_0-\tfrac{1}{2}\right) ^2} \right)   \right| \\
& \leq 
\sup_{s\in K} \sum_{j=0}^{p} 
|\xi_j^2(0) - \xi_j^2(t)|\cdot \\
&\cdot \left| \frac{1}{
	\left( \xi_j^2(t)+
		\left( s-\tfrac{1}{2}\right) ^2 
	\right)
	\left( \xi_j^2(0)+
		\left( s-\tfrac{1}{2}\right) ^2 
	\right)} -
	\frac{1}{
	\left( \xi_j^2(t)+
		\left( s_0-\tfrac{1}{2}\right) ^2 
	\right)
	\left( \xi_j^2(0)+
		\left( s_0-\tfrac{1}{2}\right) ^2 
	\right)}
\right| \\
&+
\sup_{s\in K} \sum_{j=p+1}^{\infty}\left| \frac{1}{\xi_j^2(t)+\left( s-\tfrac{1}{2}\right) ^2}-\frac{1}{\xi_j^2(t)+\left(s_0-\tfrac{1}{2}\right) ^2} \right| \\
&+
\sup_{s\in K} \sum_{j=p+1}^{\infty}\left| \frac{1}{\xi_j^2(0)+\left( s-\tfrac{1}{2}\right) ^2}-\frac{1}{\xi_j^2(0)+\left(s_0-\tfrac{1}{2}\right) ^2} \right| \\
&\leq 
3C \epsilon,
\end{align*}
where $C$ is a constant larger than every value of the bounded expression:
\begin{align*}
\left| \frac{1}{
	\left( \xi_j^2(t)+
		\left( s-\tfrac{1}{2}\right) ^2 
	\right)
	\left( \xi_j^2(0)+
		\left( s-\tfrac{1}{2}\right) ^2 
	\right)} -
	\frac{1}{
	\left( \xi_j^2(t)+
		\left( s_0-\tfrac{1}{2}\right) ^2 
	\right)
	\left( \xi_j^2(0)+
		\left( s_0-\tfrac{1}{2}\right) ^2 
	\right)}
\right|,
\end{align*}
for all $j >0$, $s\in K$ and $t\in [0,t_{\epsilon,p}]$. Hence the third term in the right-hand side of $(\ref{derivataLogaritmicaMeromorphicExtension})$ converges uniformly, and the conclusion follows.
\end{proof}
This theorem was the last ingredient needed in order to prove our main result:
\begin{proof}[Proof of Theorem \ref{convergenceZeta}]
For simplicity, suppose that $M$ goes through a pinching process along a single, simple geodesic $\eta$. From Lemma \ref{ZetaExponentialGrowth} one can easily see that:
\begin{align*}
\lim_{t\rightarrow 0}Z(s,(M,g_t)) \exp\left( - \frac{ \pi^2}{6l_t(\eta)} \right) = Z(s,(M,g_0)) 2^{1-2s},
\end{align*}
uniformly on compacts in the half-plane $\Re(s)>1$. We want to prove this convergence on the whole complex plane. To do so, it is enough to prove it on closed disks centred in $0$, of radius at least $r>2$ (so that it intersects the plane $\Re(s)>1$) out of which we cut out $N$ disjoint open disks (of the same radius $\rho$) centred in the $N$ poles of $\frac{Z'(s,(M,g_0))}{Z(s,(M,g_0))}$ in the given closed disk. Consider $K$ such a compact and denote $W_t(s):=Z(s,(M,g_t)) \exp\left( - \frac{ \pi^2}{6l_t(\eta)} \right)$. Clearly, $W_t$ has the same logarithmic derivative as $Z$ at time $t$. Fix $s_0$ a point in $K$ with $\Re(s_0)>1$ and take $\g$ one of the shortest paths in $K$ connecting $s_0$ with an arbitrary point $s\in K$. Clearly the length of $\g$ is shorter than $2r+N\pi \rho$. Thus:
\begin{align*}
|W_t(s)| 
= 
\left| W_t(s_0) \exp\left( \int_{\g} \frac{W_t'(z)}{W_t(z)}dz \right) \right|
\leq 
|W_t(s_0) | \exp\left( (2r+N\pi \rho) \sup_{s\in K} \left| \frac{W_t'(s)}{W_t(s)}
\right|  \right).
\end{align*}
Since $\frac{W_t'(s)}{W_t(s)}$ converges to a function which is holomorphic on $K$, for $t$ small enough, $W_t$ is bounded. With this in mind, let us consider the following Ansatz:
\begin{align*}
\left( \frac{W_t(s)}{Z(s,(M,g_0))2^{1-2s}} \right)'
=
\frac{W_t(s)}{Z(s,(M,g_0))2^{1-2s}} \left( \frac{W_t'(s)}{W_t(s)} - \frac{Z'(s,(M,g_0))}{Z(s,(M,g_0))} + 2\log(2) \right).
\end{align*}
The poles of $\frac{Z'(s,(M,g_0))}{Z(s,(M,g_0))}$ are the zeroes of $Z(s,(M,g_0))$ and since $K$ is far away from these poles, we deduce
that the right hand side in the above equality converges to $0$ uniformly as $t$ goes to $0$. Moreover, 
\begin{align*}
\lim_{t\rightarrow 0 } \frac{W_t(s_0)}{Z(s_0,(M,g_0))2^{1-2s_0}}  = 1,
\end{align*}
therefore $W_t$ converges to $Z(s,(M,g_0))$ uniformly on $K$. Now we have to deal with the poles. Consider $w_0$ such a point and $B_{w_0}(\delta)$ a disk around it. Then:
\begin{align*}
\lim_{t\rightarrow 0 } \int_{\partial B_{w_0}(\delta)} \frac{W_t'(z)}{W_t(z)}dz
=
\int_{\partial B_{w_0}(\delta)} \frac{Z'(s,(M,g_0))}{Z(s,(M,g_0))}dz
=
2\pi i \Res\left( \frac{Z'(s,(M,g_0))}{Z(s,(M,g_0))}\right)
\neq
0.
\end{align*}
Therefore, there exists $t_{w_0}>0$ such that $W_t$ has the same number of zeroes in $B_{w_0}(\delta)$ as $Z(s,(M,g_0))$ for $0\leq t \leq t_{w_0}$. Since $\delta$ can be taken arbitrarily small we get that $\lim_{t\rightarrow 0}W_t(w_0) = 0$.
\end{proof}

\end{document}